\documentclass[12pt,article,onecolumn,textrm]{article}
\usepackage[utf8]{inputenc}
\usepackage[english]{babel}
\usepackage{appendix}
\usepackage{amsmath}
\usepackage{mathrsfs}
\usepackage{amsfonts}
\usepackage{amssymb}
\usepackage{amsthm}
\usepackage{bbm}
\usepackage{color}
\usepackage{xcolor}
\usepackage{footmisc}
\usepackage{hyperref}
\hypersetup{
	pdftitle={},
	colorlinks=true,     
	linkcolor=blue,      
	citecolor=teal,      
	filecolor=black,      
	urlcolor=black        
}
\usepackage{graphicx}
\usepackage{enumerate}
\usepackage{graphicx}
\usepackage{authblk}
\usepackage{cancel}
\usepackage[normalem]{ulem}
\usepackage{cleveref}

\newtheorem{thm}{Theorem}[section]
\newtheorem{lem}[thm]{Lemma}
\newtheorem{prop}[thm]{Proposition}
\newtheorem{cor}[thm]{Corollary}
\newtheorem{defn}[thm]{Definition}

\theoremstyle{definition}
\newtheorem*{rem}{Remark}

\newcommand{\vs}{\vspace{0.2cm}}

\newcommand{\im}{\textup{im}}
\newcommand{\nn}{\mathbb{N}}
\newcommand{\zz}{\mathbb{Z}}
\newcommand{\rr}{\mathbb{R}}
\newcommand{\cc}{\mathbb{C}}
\newcommand{\vor}{\textrm{\normalfont Vor}}
\newcommand{\del}{\textrm{\normalfont Del}}

\newcommand{\dgm}{\textrm{\normalfont Dgm}}

\newcommand{\conv}{\textrm{\normalfont Conv}}

\newcommand{\lip}{\textrm{\normalfont Lip}}
\newcommand{\amp}{\textrm{\normalfont Amp}}
\newcommand{\pers}{\textrm{\normalfont pers}}
\newcommand{\Pers}{\textrm{\normalfont Pers}}
\newcommand{\dd}{\textrm{\normalfont d}}
\newcommand{\supp}{\textrm{\normalfont supp}}
\newcommand{\diam}{\textrm{\normalfont diam}}
\newcommand{\mesh}{\textrm{mesh}}

\newcommand{\rk}{\textrm{rk}}

\usepackage{geometry}
\begin{document}

\title{The self-similar evolution of stationary point processes via persistent homology}

\author[1]{Daniel Spitz\thanks{\href{mailto:spitz@thphys.uni-heidelberg.de}{spitz@thphys.uni-heidelberg.de}}}
\affil[1]{Institut f\"ur theoretische Physik, Ruprecht-Karls-Universit\"at Heidelberg, Philosophenweg 16, 69120 Heidelberg, Germany}

\author[2]{Anna Wienhard}
\affil[2]{Max-Planck-Institut f\"ur Mathematik in den Naturwissenschaften, Inselstraße 22, 04103 Leipzig, Germany}

\date{\today}
\setcounter{Maxaffil}{0}
\renewcommand\Affilfont{\itshape\normalsize}

\maketitle

\begin{abstract}
Persistent homology provides a robust methodology to infer topological structures  from point cloud data. 
Here we explore the persistent homology of point clouds embedded into a probabilistic setting, exploiting the theory of point processes. 
We introduce measures on the space of persistence diagrams and the self-similar scaling of a one-parameter family of these.
As the main result we prove a packing relation between the occurring scaling exponents.
\end{abstract}

\section{Introduction}\label{sec1}
\setcounter{footnote}{0}
Persistent homology allows one to infer topological structure from point cloud data. 
It was developed over the past two decades into a robust and versatile methodology, see for example~\cite{chazal2016structure, edelsbrunner2010computational, edelsbrunner2002topological, ghrist2008barcodes, otter2017roadmap, zomorodian2005computing}. 
In recent years there has been a lot of interest in bringing together the classical machinery from algebraic topology with probabilistic approaches, leading to the investigation of random geometric complexes and their limiting behavior, based on random fields and point processes~\cite{bobrowski2012euler, bobrowski2017maximally, bobrowski2017vanishing, edelsbrunner2017expected, hiraoka2018, kahle2011random, owada2017limit, yogeshwaran2015topology, yogeshwaran2017random}, see~\cite{kahle2014topology,bobrowski2018topology} for surveys. 
In particular, in~\cite{chazal2018density} the density of expected persistence diagrams and its kernel based estimation have been discussed.

In this manuscript we follow a similar route. 
We consider point processes and their persistence diagrams. We describe measures on the space of persistence diagrams. 
We show that under appropriate ergodicity assumptions {\em persistence diagram expectation measures} exist. 
Persistence diagram expectation measures are maps from the set of bounded Borel sets in $\mathbb{R}^n$ to the space of Radon measures on the set $\Delta = \{ b,d \in \mathbb{R}^2\, |\, b<d\}$. 

Using persistence diagram expectation measures, we introduce the notion of self-similar scaling for a one-parameter family of such measures.

\vs
\begin{defn}
Let $(\mathfrak{p}(t))_{t\in (T_0,T_1)}$ be a family of non-zero persistence diagram expectation measures, $0<T_0<T_1$. 
For $t\in (T_0,T_1)$ and $A\subset \rr^n$ a bounded Borel set, set $\mathfrak{p}(t,A):=\mathfrak{p}(t)(A)$. 
We say that $(\mathfrak{p}(t))_{t\in (T_0,T_1)}$ \emph{scales self-similarly between $T_0$ and $T_1$ with exponents $\eta_1,\eta_2\in \rr$}, if for all $t,t'\in (T_0,T_1)$ and $B$ an element of the Borel $\sigma$-algebra of $\Delta$:
\begin{equation*}
\mathfrak{p}(t,A)(B) = (t/t')^{-\eta_2} \mathfrak{p}(t',A)((t/t')^{-\eta_1}B),
\end{equation*}
where $\kappa B:=\{(\kappa b,\kappa d)\,|\, (b,d)\in B\}$ for $\kappa\in [0,\infty)$.
\end{defn} 
\vs

For a one-parameter family of persistence diagram expectation measures that scales self-similarly we prove a packing relation between the exponents $\eta_1$ and $\eta_2$, using results for bounded total persistence~\cite{Cohen-Steiner2010}.
Intuitively, if on average persistent homology classes grow, then the representing cycles grow as well, and less of them fit into a constant volume.
\vs

\begin{thm}\label{thm_one}
Let $(\xi(t))_{t\in (T_0,T_1)}$, $0<T_0<T_1$, be a family of stationary and ergodic simple point processes on $\rr^n$ having all finite moments. 
Let $(\mathfrak{p}(t))_t$ be the family of persistence diagram expectation measures associated to  $(\xi(t))_t$. 
Assume that all $(\mathfrak{p}(t))_t$ exist, are non-zero and that the family scales self-similarly between $T_0$ and $T_1$ with exponents $\eta_1,\eta_2\in \rr$. 
Then, if the interval $(T_0,T_1)$ is sufficiently extended, almost surely
\begin{equation*}
\eta_2 = n \eta_1.
\end{equation*}
\end{thm}
\vs
In the proof of \Cref{thm_one} we specify what sufficiently extended means. 

Strong laws of large numbers for topological descriptors such as (persistent) Betti numbers are of special interest with multifarious results already established~\cite{goel2019strong, hiraoka2018, krebs2019asymptotic, owada2020limit, roycraft2020bootstrapping,thomas2021functional}. 
We introduce the notion of ergodicity in persistence, which applies to stationary and ergodic point processes as we show.
As a further byproduct we extend the strong law of large numbers for persistent Betti numbers, proven in~\cite{hiraoka2018} for asymptotically large cubes, to convex averaging sequences which fulfil a certain balancedness criterion relating mean width and volume.

The motivation for this work originates from an application of the concept of self-similar scaling for time-dependent persistence diagrams to quantum physical systems in \cite{spitz2020finding}. 
This paper provides a more rigorous mathematical framework for some of the heuristic arguments in~\cite{spitz2020finding}.

The paper is structured as follows. 
In \Cref{SecSelfSimPackingRel} we embed persistence diagrams into the framework of point processes, in order to define persistence diagram measures. 
This leads to the introduction of self-similar scaling for one-parameter families of persistence diagram expectation measures and the formulation of the packing relation.
The latter is proven in \Cref{SecProofPackingRel}, in which we first discuss the existence of persistence diagram expectation measures and of limiting volume-averaged Radon measures.
The related strong law of large numbers for persistent Betti numbers follows.
We show ergodicity in persistence and investigate a range of geometric quantities, which then appear in the actual proof of the packing relation.
\Cref{SecExamples} is devoted to three examples constructed from different point processes. 
Finally, in \Cref{SecFurtherQuestions} we mention further questions.
\vs

\begin{rem}
We derive our results for \v{C}ech complexes, though they are easily extendable to other types of simplicial complexes. Under the assumption of point clouds being in general position, the filtrations of \v{C}ech complexes and of alpha and wrap complexes have isomorphic persistent homology groups~\cite{bauer2017morse}. Thus, all results developed in this work can be directly extended to these.%
\footnote{The proof of \Cref{LemmaErgodicPersistence} actually employs this argument.}

In~\cite{hiraoka2018} a fairly general family of complexes defined through specific measurable functions on finite sets of points is introduced, covering both the filtrations of \v{C}ech complexes and of Vietoris-Rips complexes. Care is required regarding the extension of our results to this family since for instance \Cref{LemmaPackingLemma} only applies to filtered complexes with persistent homology groups isomorphic to those of the filtration of \v{C}ech complexes.
\end{rem}

\section{Self-similarity and the packing relation}\label{SecSelfSimPackingRel}
The generation of point clouds can be embedded into the probabilistic setting via the theory of point processes, necessary statements of which we review in \Cref{SecPointProcesses}.
In \Cref{SecPersDiagMeasures} we briefly introduce persistent homology and persistence diagram expectation measures, which will allow for the probabilistic treatment of persistence diagrams.
This facilitates the formulation of self-similar scaling of a one-parameter family of such persistence diagram expectation measures, given in \Cref{SecSelfSimilar}.
Corresponding scaling exponents are linked by a packing relation, which we finally state there.

\subsection{Point processes}\label{SecPointProcesses}
Point processes are random collections of at most countably many points, identified with random counting measures. We review notions from the theory of point processes~\cite{daley2003introduction, daley2007, last2017lectures}.

Let $S$ be a complete and separable metric space, $\mathcal{B}(S)$ its Borel $\sigma$-algebra and $\mathcal{R}(S)$ the space of Radon measures%
\footnote{A measure $\mu$ on $(S,\mathcal{B}(S))$ is called a \emph{Radon measure}, if $\mu(A)<\infty$ for every relatively compact $A\in\mathcal{B}(S)$. In particular, since $S$ is a complete, separable metric space, a thus defined Radon measure is regular~\cite{daley2003introduction}.} 
on $S$.
We define $\mathcal{B}^n=\mathcal{B}(\rr^n)$ and denote the set of bounded Borel subsets of $\rr^n$ by $\mathcal{B}_b^n$. 
The space of all boundedly finite measures on $\rr^n$ is denoted by $\mathcal{M}^n$ and of the integer-valued ones by $\mathcal{N}^n\subset \mathcal{M}^n$.
We denote the $n$-dimensional Lebesgue measure by $\lambda_n$. 
  
Let $(\Omega, \mathcal{E}, \mathbb{P})$ be a probability space. A \emph{random measure} on $\rr^n$ is a measurable map $\mu:(\Omega,\mathcal{E}, \mathbb{P})\to (\mathcal{M}^n, \mathcal{B}(\mathcal{M}^n))$. A \emph{point process} on $\rr^n$ is a measurable map $\xi: (\Omega,\mathcal{E},\mathbb{P})\to (\mathcal{N}^n, \mathcal{B}(\mathcal{N}^n))$. 

A point process $\xi$ is \emph{simple}, if a.s. all its atoms are distinct; it \emph{has all finite moments}, if for all $k\geq 1$ and $A\in \mathcal{B}^n_b$: $\mathbb{E}[\xi(A)^k] < \infty$. 
We assume that all point processes studied in this work have non-zero first moments, $\mathbb{E}[\xi(A)] > 0$ for $A\in\mathcal{B}_b^n$ with $\lambda_n(A)>0$. 
For a simple point process $\xi$ on $\rr^n$ we define the set of atoms in $A\in\mathcal{B}^n_b$ for a sample $\omega\in\Omega$ as
\begin{equation*}
X_{\xi_\omega} (A) = \left\{ x_i\, \left|\, \xi_\omega(A) = \sum_i \delta_{x_i}(A)\right.\right\}. 
\end{equation*}
The sets $X_{\xi_\omega} (A)$ a.s. form point clouds for $A\in \mathcal{B}_b^n$, i.e., finite subsets of $\rr^n$.
When referring to $X_{\xi_\omega}(A)$ as point clouds, we omit potential multiplicities of atoms, which can occur with probability zero.
We often omit the sample $\omega$ from notations.

A random measure $\mu$ on $\rr^n$ is \emph{stationary}, if for all $x\in \rr^n$ the distributions of $\mu$ and $\theta_x \mu$ coincide, where $(\theta_x \mu)(A):= \mu(A+x)$ for all $A\in \mathcal{B}^n_b$. 
A stationary random measure $\mu$ is \emph{ergodic}, if $\mathbb{P}(\mu\in \mathcal{A})\in \{0,1\}$ for all Borel sets $\mathcal{A}$ of $\mathcal{M}^n$ that are invariant under translation by $x$ for all $x\in \rr^n$. 

A sequence $\{A_k\}\subset \mathcal{B}^n_b$ is a \emph{convex averaging sequence} if
\begin{align*}
&\textrm{(i) each $A_k$ is convex,\qquad (ii) $A_k\subseteq A_{k+1}$ for all $k$, \qquad(iii) $\lim_{k\to\infty} r(A_k)=\infty$},
\end{align*}
where $r(A) := \sup \{r \, |\, A \textrm{ contains a ball of radius }r\}$.
\vs

\begin{prop}[Corollary 12.2.V in  \cite{daley2007}]\label{PropErgodicityMeanDensity}
Let $\{A_k\}$ be a convex averaging sequence. 
If a random measure $\xi$ on $\rr^n$ is stationary, ergodic and has finite expectation measure with mean density $m = \mathbb{E}[\xi([0,1]^n)]=m\lambda_{n}([0,1]^n)$, then $\xi(A_k)/\lambda_n(A_k) \to m$ a.s. as $k\to\infty$.
\end{prop}

\subsection{Persistence diagram expectation measures}\label{SecPersDiagMeasures}
We briefly review notions from persistent homology before defining persistence diagram expectation measures.
For a general introduction to algebraic topology we refer to~\cite{hatcher2005, may1999concise, munkres1984elements}; for a thorough introduction to computational topology the interested reader may consult~\cite{edelsbrunner2010computational, otter2017roadmap}.

For a given point cloud $X\subset \rr^n$ we consider its filtration of \v{C}ech complexes $(\check{C}_r(X))_r$.%
\footnote{The \v{C}ech complex of $X$ of radius $r$ is defined as the abstract simplicial complex $\check{C}_r(X)=\{\sigma\subseteq X\,|\, \bigcap_{x\in\sigma}B_r(x)\neq \emptyset\}$.} 
We denote the $m$-skeleton of $\check{C}_r(X)$ by $\check{C}_r(X)^m$, which consists of all simplices of $\check{C}_r(X)$ up to and including dimension $m$.
The map of homology groups induced by the inclusion $\check{C}_r(X)\hookrightarrow \check{C}_s(X)$, $r\leq s$, is denoted by $\iota^{r,s}_\ell:H_\ell(\check{C}_r(X))\to H_\ell(\check{C}_s(X))$, where we consider homology with coefficients in an arbitrary field $\mathbb{F}$. 
The \emph{$\ell$-th persistent Betti numbers} are defined as $\beta^{r,s}_\ell(\mathcal{C}(X)):= \rk (\im (\iota_\ell^{r,s}))$.

We set $\Delta:=\{(b,d)\in\rr^2\,|\, b<d\}$ and denote the space of persistence diagrams by $\mathscr{D}$, which is the space of finite multisets of points in $\Delta$.
By finiteness of $X$ the persistent homology of the filtration $\{\check{C}_r(X)\}$ is fully described by its persistence diagram $\dgm_\ell(X)$.%
\footnote{The persistence module $H_\ell(\mathcal{C}(X))=(H_\ell(\check{C}_r(X)), \iota^{r,s}_\ell)_{s\geq r}$ is tame, such that the structure theorem~\cite{zomorodian2005computing} applies, which yields the isomorphism of $H_\ell(\mathcal{C}(X))$ with its persistence diagram.}  
A persistence pair $(b,d)\in\dgm_\ell(X)$ corresponds to an $\ell$-dimensional homology class being present in the \v{C}ech complexes for all radii in $[b,d)$, with $b$ its birth radius, $d$ its death radius and $\pers((b,d)):=d-b$ its persistence. 
We ignore persistence pairs with zero persistence in persistence diagrams. 
While this is consistent for our work, persistence pairs with zero persistence can be important e.g. for the computation of the bottleneck distance.

In the probabilistic setting we can consider persistence diagrams for point clouds generated by point processes.
For this let $\xi$ be a simple point process on $\rr^n$. 
Then for any $A\in \mathcal{B}^n_b$ and $\omega\in\Omega$ we have the point cloud $X_{\xi_\omega}(A)$. 
We denote by $D_A( X_{\xi_\omega}(A)): = \bigcup_{\ell=0}^{n-1} \dgm_\ell(X_{\xi_\omega}(A))$ the corresponding persistence diagram of the filtration of the \v{C}ech complexes. 
This leads to point processes on the space of persistence diagrams $\mathscr{D}$, with persistence diagrams of individual point cloud samples as atom sets.
\vs

\begin{defn}
Given a simple point process $\xi$ and $A\in\mathcal{B}_b^n$, the map 
\begin{equation*}
\rho_\omega(A):=\sum_{x\in D_A( X_{\xi_\omega}(A))} \delta_{x}\qquad \textrm{for all }\omega\in\Omega
\end{equation*}
defines a point process $\rho_\cdot(A)$ on $\Delta$. 
The map $\rho_\cdot: \Omega\times \mathcal{B}^n_b\to \mathcal{N}(\Delta)$ is called a \emph{persistence diagram measure}. 
If it exists, its first moment measure defines
\begin{equation*}
\mathfrak{p}(A):= \mathbb{E}[\rho_\omega(A)]. 
\end{equation*}
The map $\mathfrak{p}: \mathcal{B}^n_b \to \mathcal{R}(\Delta)$, $A\mapsto \mathfrak{p}(A)$ is called a \emph{persistence diagram expectation measure}.
\end{defn}

\subsection{Self-similar scaling and the packing relation}\label{SecSelfSimilar}
We can now introduce the notion of self-similar scaling for a one-parameter family of persistence diagram expectation measures. 
The additional parameter can be interpreted, for instance, as the time in physics applications~\cite{spitz2020finding}.
Self-similar scaling describes a power-law parameter dependence of persistence diagram expectation measures.
One of the exponents, $\eta_1$, encodes the time-dependence of persistence length scales, while the other, $\eta_2$, describes how the overall number of persistence pairs behaves.
More precisely, this interpretation is based on \Cref{LemmaScalingBehavior}.
\vs

\begin{defn}\label{DefiScaling}
Let $(\mathfrak{p}(t))_{t\in (T_0,T_1)}$ be a family of non-zero persistence diagram expectation measures, $0<T_0<T_1$. 
For $t\in (T_0,T_1)$, $A\in\mathcal{B}^n_b$ set $\mathfrak{p}(t,A):=\mathfrak{p}(t)(A)$ and let $\{A_k\}$ be a convex averaging sequence. 
We say that $(\mathfrak{p}(t))_{t\in (T_0,T_1)}$ \emph{scales self-similarly between $T_0$ and $T_1$ with exponents $\eta_1,\eta_2\in \rr$}, if for all $t,t'\in (T_0,T_1)$, $B\in \mathcal{B}(\Delta)$ and $k$ sufficiently large depending on the sequence $(A_k)$,
\begin{equation*}
\mathfrak{p}(t,A_k)(B) = (t/t')^{-\eta_2} \mathfrak{p}(t',A_k)((t/t')^{-\eta_1}B),
\end{equation*}
where $\kappa B:=\{(\kappa b,\kappa d)\,|\, (b,d)\in B\}$ for $\kappa\in [0,\infty)$.
\end{defn}
\vs

A relation between the scaling exponents $\eta_1$ and $\eta_2$ can be established.
Intuitively, it is based on the bounded packing of cycles representing persistent homology classes, less of which fit into a given constant volume if persistence length scales grow.
\vs

\begin{thm}[The packing relation]\label{ThmPackingRelation}
Let $(\xi(t))_{t\in (T_0,T_1)}$, $0<T_0<T_1$, be a family of stationary and ergodic simple point processes on $\rr^n$ having all finite moments. 
Let $(\mathfrak{p}(t))_t$ be the family of persistence diagram expectation measures computed from $(\xi(t))_t$. 
Assume that all $(\mathfrak{p}(t))_t$ exist, are non-zero and that the family scales self-similarly between $T_0$ and $T_1$ with exponents $\eta_1,\eta_2\in \rr$. 
Then, if the interval $(T_0,T_1)$ is sufficiently extended as detailed in the proof, we a.s. find
\begin{equation*}
\eta_2 = n \eta_1.
\end{equation*}
\end{thm}
\vs
The lengthier proof of this theorem is deferred to \Cref{SecProofPackingRel}.
The packing relation can have applications throughout many-body physics whenever self-similar scaling with some parametric dependence occurs for persistent homology quantifiers, e.g. in persistent homology studies of critical scaling phenomena and non-thermal fixed points~\cite{spitz2020finding,Sale:2022qfn,sale2022quantitative,Sehayek:2022lxf}.

\section{Proof of the packing relation}\label{SecProofPackingRel}
This section is devoted to the proof of the packing relation, which proceeds in a number of steps.
First, we show in \Cref{SecExistencePersDiagExpecMeasure} that persistence diagram expectation measures can indeed exist.
Then, in \Cref{SecLimitingRadonMeasures} an extension of the existence of limiting Radon measures for volume-averaged persistence diagram (expectation) measures to so-called balanced convex averaging sequences is proven, given in~\cite{hiraoka2018} for averaging along cubes.
As a corollary the related extension of the strong law of large numbers for persistent Betti numbers of~\cite{hiraoka2018} follows.
Subsequently, in \Cref{SecErgodicityPers} we show a persistent homology variant of ergodicity for stationary and ergodic point processes to generate point clouds.
In \Cref{SecGeometricQuantities} we define a number of geometric quantities for persistence diagram (expectation) measures, which show up in the proof of the packing relation. 
The latter is finally presented in \Cref{SecProofPackingRelFinal} and makes use of the previous results.

\subsection{Existence of persistence diagram expectation measures}\label{SecExistencePersDiagExpecMeasure}
Persistence diagram expectation measures exist for stationary and ergodic point processes on $\rr^n$, as the following lemma shows.
\vs

\begin{lem}\label{LemmaExistencePersDiagExpecMeasure}
Let $\xi$ be a simple stationary and ergodic point process on $\rr^n$ having all finite moments. Then the corresponding persistence diagram expectation measure exists.
\end{lem}
\vs

\begin{proof}
Let $A\in \mathcal{B}^n_b$ and $B\in \mathcal{B}(\Delta)$ be bounded, $\omega\in\Omega$. Let $F_\ell(X_{\xi_\omega}(A),r)$ be the number of $\ell$-simplices in $\check{C}_r(X_{\xi_\omega}(A))$ and $F_\ell(\xi,r;A)$ the number of $\ell$-simplices in $\check{C}_r(X_{\xi_\omega}(\rr^n))$ with at least one vertex in $A$. 
Let $\# K$ denote the number of simplices in a simplicial complex $K$.
Using the $n$-skeleton of the \v{C}ech complex, $\check{C}_r(X_{\xi_\omega}(A))^n$, we compute for $r\geq 0$ sufficiently large,
\begin{align*}
\rho_\omega(A)(B) \leq n(X_{\xi_\omega}(A)) \leq&\; \# (\check{C}_r(X_{\xi_\omega}(A))^n)\\
=&\; \sum_{\ell=0}^n F_\ell(X_{\xi_\omega}(A),r) \\
\leq&\; \sum_{\ell=0}^n F_\ell(\xi, r; A)\leq \sum_{\ell=0}^n \sum_{x\in X_{\xi_\omega}(A)} \binom{\xi(B_{2\ell r}(x))}{\ell + 1},
\end{align*}
which holds since all vertices of an $\ell$-simplex containing a vertex $x\in A$ are contained in $X_{\xi_\omega}(\rr^n)\cap B_{2\ell r}(x)$, and all $\ell$-simplices are constructed from $\ell+1$ points in this intersection.
If $\ell + 1 > \xi(B_{2\ell r}(x))$, the binomial coefficient is zero.
Let $R(\xi_\omega,A)<\infty $ be the maximal $r>0$, such that for any $\epsilon>0$, $\epsilon < r$:
\begin{equation*}
\check{C}_r(X_{\xi_\omega}(A)) \neq \check{C}_{r-\epsilon}(X_{\xi_\omega}(A)).
\end{equation*}
Indeed, $R(\xi_\omega,A)$ exists, since the point clouds $X_{\xi_\omega}(A)$ are always finite, hence the \v{C}ech complex $\check{C}_r(X_{\xi_\omega}(A))$ changes only finitely often as $r$ increases.
Then we find
\begin{equation*}
\rho_\omega(A)(B) \leq n(X_{\xi_\omega}(A)) \leq \#(\check{C}_\infty (X_{\xi_\omega}(A))^n)\leq \sum_{\ell=0}^n \sum_{x\in X_{\xi_\omega}(A)} \binom{\xi(B_{2\ell R(\xi_\omega,A)}(x))}{\ell +1}.
\end{equation*}
Taking expectations yields
\begin{equation*}
\mathfrak{p}(A)(B) \leq \sum_{\ell=0}^n \mathbb{E}\left[ \sum_{x\in X_{\xi_\omega}(A)} \binom{\xi(B_{2\ell R(\xi_\omega,A)}(x))}{\ell + 1} \right].
\end{equation*}
The expectation value on the right-hand side of this expression exists and is finite for bounded $A$, since $\xi$ is stationary, ergodic and has all finite moments. 
Thus, $\mathfrak{p}(A)(B)$ is finite and in particular exists.
\end{proof}

Concerning the structure of persistence diagram (expectation) measures, we note the following.
\vs

\begin{prop}\label{PropExistenceLebesgueDensities}
Let $\rho$ be a persistence diagram measure, $\mathfrak{p}$ the corresponding persistence diagram expectation measure and $A\in \mathcal{B}^n_b$, $\omega\in\Omega$. 
Then, there exist Borel measurable functions $\tilde{\rho}_{\omega,A},\tilde{\mathfrak{p}}_A:\Delta\to \rr_+$ and singular measures $\rho_{\omega,s}(A), \mathfrak{p}_s(A)$, such that for all $B\in \mathcal{B}(\Delta)$:
\begin{subequations}
\begin{align*}
\rho_\omega(A)(B) = &\; \int_B \tilde{\rho}_{\omega,A}(x) \lambda_2(\dd x) + \rho_{\omega,s}(A)(B),\\
\mathfrak{p}(A)(B) = &\; \int_B \tilde{\mathfrak{p}}_A(x) \lambda_2(\dd x) + \mathfrak{p}_s(A)(B),
\end{align*}
\end{subequations}
$\lambda_2$ the Lebesgue measure on $\Delta$.
\end{prop}
\begin{proof}
This is a direct consequence of applying first the Lebesgue decomposition theorem and then the Radon-Nikodym theorem for absolutely continuous measures~\cite{mattila1999geometry}. 
Focussing on $\mathfrak{p}(A)$, we explicitly set
\begin{equation*}
\tilde{\mathfrak{p}}_A(x):=\lim_{r \searrow 0} \frac{\mathfrak{p}(A)(B_r(x))}{\lambda_2(B_r(x))}
\end{equation*}
and
\begin{equation*}
S:=\left\{ x : \lim_{r \searrow 0} \frac{\mathfrak{p}(A)(B_r(x))}{\lambda_2(B_r(x))} = \infty\right\}.
\end{equation*}
Then we set $\mathfrak{p}_s(A)(B)=\mathfrak{p}(A)(B\cap S)$ and find $\lambda_2(S)=0$. 
Fully analogous constructions lead to the decomposition of $\rho_\omega(A)$.
\end{proof}

\begin{rem}
In~\cite{chazal2018density} it is shown that under particular assumptions on the simplicial complex filtration, covering for instance the filtration of \v{C}ech complexes, the singular contribution $\mathfrak{p}_s$ to the persistence diagram expectation measure $\mathfrak{p}$ can be absorbed into the density $\tilde{\mathfrak{p}}$.
\end{rem}
\vs

\begin{rem}
The measurable Lebesgue density corresponding to the non-singular contribution to the persistence diagram (expectation) measure appearing in \Cref{PropExistenceLebesgueDensities} can be identified with the so-called (asymptotic) persistence pair distribution defined in~\cite{spitz2020finding}. 
In particular, if the persistence diagram (expectation) measure exists, then the density exists as well, though it might be zero.
\end{rem}

\subsection{Existence of limiting volume-averaged Radon measures}\label{SecLimitingRadonMeasures}
In this section we state an extension of the existence of large-volume averages of persistence diagram expectation measures to so-called balanced convex averaging sequences, established in~\cite[Theorem 1.11]{hiraoka2018} for averaging along cubes. 
The strong law of large numbers for persistent Betti numbers can be extended analogously.
The proofs are given in \Cref{AppendixProofsLimitTheorems}. 
They make use of the results for cubical averaging sequences~\cite{hiraoka2018} and convex geometry.
\vs

\begin{defn}
A convex averaging sequence $\{A_k\}$ is called \emph{balanced}, if
\begin{equation*}
W_{n-1}(A_k)=O(\lambda_n(A_k)^{1/n})\textrm{ for sufficiently large }k,
\end{equation*}
where $W_{n-1}$ denotes the $(n-1)$-st quermassintegral.
The quermassintegrals $W_i(C)$, $i=0,\dots,n$, of a convex body $C$ in $\rr^n$ are given by \cite{gruber2007}
\begin{equation}\label{EqQuermassintegralDef}
W_i(C) = V(\underset{n-i}{\underbrace{C,\dots,C}},\underset{i}{\underbrace{B_1(0),\dots,B_1(0)}}),
\end{equation}
where $V$ denotes the mixed volume of $n$ convex bodies in $\rr^n$ and $B_r(x)$ denotes the closed ball of radius $r$ around $x\in\rr^n$. 
They are further used in \Cref{AppendixProofsLimitTheorems}.
\end{defn}
\vs

\begin{rem}
Balancedness of a convex averaging sequence is similar to the vanishing relative boundary condition for sequences of Borel subsets of $\rr^n$ in~\cite{penrose2001central}.
To be balanced is a fairly general condition. 
For instance, sequences of growing cubes, simplices, polyhedra or balls fulfill this property. 
Balancedness is required in the proof of \Cref{ThmExistenceLimitingRadonMeasure} in order to exclude sequences of increasingly flat convex bodies for which the ratio of the $n$-th power of the mean width to volume does not converge to zero for $k\to\infty$~\cite{ball1991volume}.
\end{rem}
\vs

The following theorem establishes the existence of a limiting Radon measure, towards which volume-averaged persistence diagram expectation measures for balanced convex averaging sequences converge.
\vs
\begin{thm}\label{ThmExistenceLimitingRadonMeasure}
Let $\xi$ be a simple point process on $\rr^n$ having all finite moments and $\{A_k\}$ a balanced convex averaging sequence. 
If $\xi$ is stationary, then there exists a unique Radon measure $\mathfrak{P}\in \mathcal{R}(\Delta)$, such that
\begin{equation*}
\frac{\mathfrak{p}(A_k)}{\lambda_n(A_k)}\overset{v}{\longrightarrow} \mathfrak{P}\quad \textrm{ for }k\to\infty,
\end{equation*}
where $\mathfrak{p}$ is the persistence diagram expectation measure corresponding to $\xi$ and $\overset{v}{\to}$ denotes vague convergence.
\end{thm}
\vs

The existence of limiting Radon measures for general balanced convex averaging sequences is of interest for applications in the natural sciences, where large-volume asymptotics are not only taken via nested sequences of cubes.
An example is given by thermodynamic limits of spherically shaped systems in quantum many-body physics~\cite{Gauthier:2018skd,doi:10.1126/science.aat5793}.

Similar to \Cref{ThmExistenceLimitingRadonMeasure}, a stronger statement can be established for persistent Betti numbers.
\vs

\begin{cor}[Strong law of large numbers for persistent Betti numbers]\label{CorStrongLawPersBetti}
Let $\xi$ be a stationary simple point process having all finite moments and $\{A_k\}$ a balanced convex averaging sequence. 
Given $\omega\in\Omega$, we set $X_k:=X_{\xi_\omega}(A_k)$. 
Then, for any $0\leq r\leq s < \infty$ and $\ell =0,1,\dots, n$ there exists a constant $\hat{\beta}^{r,s}_\ell$, such that
\begin{equation*}
\frac{\mathbb{E}[\beta^{r,s}_\ell (\mathcal{C}(X_k))]}{\lambda_n(A_k)} \to \hat{\beta}^{r,s}_\ell \qquad \textrm{ for } k\to\infty.
\end{equation*}
Additionally, if $\xi$ is ergodic, then a.s.
\begin{equation*}
\frac{\beta^{r,s}_\ell (\mathcal{C}(X_k))}{\lambda_n(A_k)} \to \hat{\beta}^{r,s}_\ell \qquad \textrm{ for } k\to\infty.
\end{equation*}
\end{cor}
\vs

In the spirit of \Cref{ThmExistenceLimitingRadonMeasure} and \Cref{CorStrongLawPersBetti}, several related results have been established for particular point processes. 
For instance, for classes of binomial and Poisson point processes the strong law of large numbers for Betti numbers in the thermodynamic regime~\cite{goel2019strong} and the asymptotic normality of persistent Betti numbers~\cite{krebs2019asymptotic} have been shown. 
In~\cite{owada2020limit} limit theorems have been established for Poisson processes in sparse, Poisson and critical regimes.
A bootstrapping procedure for persistent Betti numbers has been developed in~\cite{roycraft2020bootstrapping}, and~\cite{thomas2021functional} considers the Euler characteristic for the Poisson process in the critical regime.

\subsection{Ergodicity in persistence}\label{SecErgodicityPers}
We define a notion of ergodicity for point processes on the space of persistence diagrams.
\vs
\begin{defn}
Let $\xi$ be a stationary simple point process on $\rr^n$ with finite expectation measure, $\{A_k\}$ a convex averaging sequence, and $n(X_k):=\# \bigcup_{\ell=0}^{n-1} \dgm_\ell (X_k)$, $X_k:=X_\xi(A_k)$, for all $k$. 
We say that $\xi$ is \emph{ergodic in persistence} if a.s. $n(X_k)/\lambda_n(A_k)$ converges to a finite number as $k\to\infty$.
\end{defn}
\vs

One can establish that stationary and ergodic point processes lead to ergodicity in persistence.
\vs

\begin{lem}\label{LemmaErgodicPersistence}
Let $\xi$ be a simple point process on $\rr^n$ having all finite moments. 
If $\xi$ is stationary and ergodic, then $\xi$ is ergodic in persistence.
\end{lem}
\vs

The proof of \Cref{LemmaErgodicPersistence} makes use of two auxiliary Lemmas, which we show first.
\vs

\begin{lem}\label{LemnDelaunay}
Let $\xi$ be a simple and stationary point process on $\rr^n$ having all finite moments, $A\in\mathcal{B}^n_b$.
Then a.s.
\begin{equation*}
n(X_\xi(A)) = n_{\del}(X_\xi(A)),
\end{equation*}
with $n_{\del}(X_k)$ the number of persistent homology classes computed for the family of Delaunay complexes.%
\footnote{The Delaunay complex of a point cloud $X\subset \rr^n$ is defined as $\del_r(X)=\{Q\subseteq X\, |\, \bigcap_{x\in Q}\vor_r(x,X)\neq \emptyset\}$ with $\vor_r(x,X)=B_r(x)\cap \{y\in\rr^n\,|\, |y-x|\leq |y-p|\,\forall\,p\in X\}$ the Voronoi ball around $x\in X$~\cite{bauer2017morse}.}
\end{lem}

\begin{proof}
Following~\cite{bauer2017morse}, $X_\xi(A)$ is in general position if for every $P\subseteq X_\xi(A)$ of at most $n+1$ points (i) $P$ is affinely independent, and (ii) no point of $X_\xi(A)\setminus P$ lies on the smallest circumsphere of $P$.
Given that $\xi$ is stationary, its first moment measure is proportional to the Lebesgue measure~\cite{daley2007}, i.e., for $A\in \mathcal{B}^n_b$:
\begin{equation}\label{EqFirstMomentMeasureLebesgueIntensity}
\mathbb{E}[\xi(A)] = \mathbb{E}[\xi([0,1]^n)] \lambda_n(A).
\end{equation}
For $i<n$ any $i$-dimensional hyperplane and any $i$-dimensional hypersphere is a subset of $\rr^n$ of Lebesgue measure zero. 
It follows a.s. that the affine hull of any subset of $n+1$ points of a sample $X_\xi(A)$ is the entire%
\footnote{Else, with non-vanishing probability a Lebesgue-zero set would contain one or more points, in contradiction to \Cref{EqFirstMomentMeasureLebesgueIntensity}.} 
$\rr^n$, showing (i).
Let $P\subset X_\xi(A)$ be a subset consisting of $j \leq n+1$ points. 
$P$ has a $(j-2)$-dimensional circumsphere with minimal radius, on which a.s. no point of $X_\xi(A)\setminus P$ lies, given that circumspheres are Lebesgue-zero sets, showing (ii).
$X_\xi(A)$ thus being a.s. in general position, its \v{C}ech and Delaunay complex filtrations a.s. have isomorphic persistent homology~\cite{bauer2017morse}.
This proves the claim.
\end{proof}

A second auxiliary lemma shows that the size of the Delaunay triangulation $\del_\infty(X_k)$ scales proportional to $\#X_k$ in the limit of large $A_k$, based on results of~\cite{Amenta2012}.
\vs
\begin{prop}[Theorem 1 of~\cite{Amenta2012}]\label{PropDelaunayTriangulationScaling}
Let $X$ be a $\lambda$-sparse $\epsilon$-sample of a $p$-dimensional polyhedron $P$ in $\rr^n$, i.e., (i) every point $x\in P$ has a distance $\epsilon$ or less to a point in $X$, and (ii) every closed $n$-ball with radius $5n\epsilon$ contains at most $\lambda$ points of $X$.
In the worst case, the Delaunay triangulation of $X$ has size $\Theta((\#X)^{\frac{d-k+1}{p}})$, where $k=\lceil \frac{d+1}{p+1}\rceil$.
\end{prop}
\vs

\begin{lem}\label{LemDelaunayScaling}
Let $\xi$ be a simple and ergodic point process on $\rr^n$ having all finite moments.
Then a.s. $\#\del_\infty(X_k) = \Theta(\#X_k)$.
\end{lem}

\begin{proof}
We a.s. show the assumptions of \Cref{PropDelaunayTriangulationScaling}, i.e., that there exist $\lambda,\epsilon\in\rr$ such that the $X_k$ are $\lambda$-sparse $\epsilon$-samples of $\conv(X_k)$, the convex hull of $X_k$, which is a polyhedron in the sense of \cite{Amenta2012}.
Then $\#\del_\infty(X_k) = \Theta(\#X_k)$.

To show (i) assume there exists with non-zero probability $x\in \conv(X_k)$ such that for all $\epsilon >0$, $s\in X_k$: $|x-s|>\epsilon$. 
Then, with non-zero probability a ball of arbitrary radius $r>0$ around $x$ exists, such that $\xi(B_r(x))=0$. 
With non-zero probability this then leads for a sample $\xi$ and $r\to\infty$ to 
\begin{equation*}
\frac{\xi(B_r(x))}{\lambda_n(B_r(x))}\to 0.
\end{equation*}
Since $\xi$ is ergodic, via \Cref{PropErgodicityMeanDensity} we can conclude with non-zero probability for such a sample $\xi$ that $\mathbb{E}[\xi([0,1]^n)] = 0$, in contradiction to the assumption of $\xi$ having non-zero first moment. 
Thus, an $\epsilon$ exists such that a.s. for all $x\in \conv(X_k)$ there exists $s\in X_k$: $|x-s|<\epsilon$; a.s. (i) holds.

To show (ii) assume that with non-zero probability there exists an $n$-ball $B$ with radius $5n\epsilon$, such that for all $\lambda>0$: $\xi(B)>\lambda$. 
Then, $\xi(B)=\infty$ with non-zero probability, in contradiction to $\xi$ having all finite moments. 
Thus, every closed $n$-ball with radius $5n\epsilon$ a.s. contains at most $\lambda$ points of $X_k$, such that a.s. (ii) holds.
Indeed, the assumptions of Theorem 1 of \cite{Amenta2012} are fulfilled and $\#\del_\infty(X_k) = \Theta(\#X_k)$.
\end{proof}

We can now prove \Cref{LemmaErgodicPersistence}.

\begin{proof}[Proof of \Cref{LemmaErgodicPersistence}]
Due to \Cref{PropErgodicityMeanDensity} for any $\varepsilon>0$ there exists $N\in\nn$ such that a.s. for $k > N$:
\begin{equation}\label{EqScalingErgodicityOntheWay}
\left|  \frac{\#X_k}{\lambda_n(A_k)} - m\right| < \varepsilon,
\end{equation}
where $m = \mathbb{E}[\xi([0,1]^n)]$ is the mean density.
Next, we show that $n(X_k)/\# X_k$ converges as $k\to\infty$.
By \Cref{LemnDelaunay} we have $n(X_k)=n_{\del}(X_k)$.
Denoting the $\ell$-th persistence diagram of the filtration of Delaunay complexes $(\del_r(X_k))_r$ by $\dgm_{\ell,\del}(X_k)$, we find
\begin{equation*}
n_{\del}(X_k)=\#\bigcup_{\ell=0}^{n-1} \dgm_{\ell,\del} (X_k) \leq \sum_{\ell=0}^{n} \dim_{\mathbb{F}}\big(C_\ell(\del_\infty (X_k))\big) = \# \del_\infty(X_k),
\end{equation*}
with $C_\ell(\del_\infty(X_k))$ the $\ell$-th chain group of $\del_\infty(X_k)$ with $\mathbb{F}$-coefficients.
Now, $\#\del_\infty(X_k) = \Theta(\#X_k)$ by \Cref{LemDelaunayScaling}, such that $n_{\del}(X_k) \leq \Theta(\#X_k)$.
On the other hand, $n_{\del}(X_k)\geq \# X_k$ for all $k$, since the points of $X_k$ represent the dimension-0 persistent homology classes.
Thus, for any $\varepsilon >0$ there a.s. exist constants $C>0$ and $N'\in\nn$, such that for all $k\geq N'$:
\begin{equation}\label{EqConvergenceTo11}
\left| \frac{n(X_k)}{\# X_k} - C\right| < \varepsilon.
\end{equation}
By \Cref{EqScalingErgodicityOntheWay,EqConvergenceTo11} there exists $M$, such that a.s. for all $k\geq M$:
\begin{align*}
\left|\frac{n(X_k)}{\lambda_n(A_k)} - mC\right|\leq &\; \left| \frac{n(X_k)}{\lambda_n(A_k)} - \frac{m\, n(X_k)}{\#X_k}\right| + \left|\frac{m \, n(X_k)}{\#X_k} - m C\right|\nonumber\\
=&\; \frac{n(X_k)}{\#X_k} \left| \frac{\#X_k}{\lambda_n(A_k)} - m\right| + m \left| \frac{n(X_k)}{\#X_k} - C\right|\nonumber\\
< &\; \left(\frac{n(X_k)}{\#X_k} + m\right)  \varepsilon.
\end{align*}
The prefactor $n(X_k)/\#X_k$ a.s. converges for $k\to \infty$ by \Cref{EqConvergenceTo11}.
This yields ergodicity in persistence.
\end{proof}

\subsection{Geometric quantities}\label{SecGeometricQuantities}
Different geometric quantities may be computed from persistence diagram measures, which later show up in the proof of the packing relation (\Cref{ThmPackingRelation}).
These are of independent interest e.g. for applications in many-body physics, see \cite{spitz2020finding}.
\vs

\begin{defn}\label{DefGeometricQuantitiesPersistence}
Let $\rho$ be a persistence diagram measure and $A\in\mathcal{B}^n_b$. 
Let $\omega\in\Omega$. 
We define the \emph{number of persistent homology classes} as
\begin{equation*}
n_\omega(A):=\int_{\Delta}  \rho_\omega(A)(\dd x) = \rho_\omega(A)(\Delta).
\end{equation*}
Let $q>0$. 
The \emph{degree-$q$ persistence} is defined as
\begin{equation*}
l_{q,\omega}(A) := \left[\frac{1}{n_\omega(A)} \int_{\Delta} \pers(x)^q \, \rho_\omega(A)(\dd x)\right]^{1/q}.
\end{equation*}
With $Y_\omega(A)\subset \Delta$ the multiset of atoms of $\rho_\omega(A)$, we define the \emph{maximum death} as
\begin{equation}\label{EqMaximumDeathDef}
d_{\max,\omega}(A):= \max \{ d\, |\, (b,d)\in Y_\omega(A)\} = \lim_{p\to \infty} \left[  \int_{\Delta} d(x)^p\, \rho_\omega(A)(\dd x)\right]^{1/p},
\end{equation}
where the last equality is a general result for $p$-norms in finite dimensions.
\end{defn}
\vs

Bounded total persistence as introduced in~\cite{Cohen-Steiner2010} and described in \Cref{AppendixBoundedTotalPersistence} yields a link between the geometric quantities, which results in an upper bound for the number of persistent homology classes in a given volume.
\vs

\begin{lem}[The packing lemma]\label{LemmaPackingLemma}
Let $\xi$ be a simple point process on $\rr^n$ and let $n_\omega(\cdot), d_{\max,\omega}(\cdot)$ and $l_{q,\omega}(\cdot)$ be computed from the sample $\xi_\omega$, $\omega\in\Omega$. 
Then there exists a constant $c>0$, such that for any $\delta >0$ and $A\in \mathcal{B}^n_b$:
\begin{equation*}
n_\omega(A) \leq \frac{c\, (n+2\delta)}{\delta} \frac{d_{\max,\omega}(A)^{\delta}}{l_{n+\delta,\omega}(A)^{n+\delta}}.
\end{equation*}
\end{lem}

\begin{proof}
Let $A\in\mathcal{B}^n_b$ and $X:= X_{\xi_\omega}(A)$. 
For all $x\in \conv(X)$ we set 
\begin{equation*}
d_X(x):=\min_{p\in X} d(x,p).
\end{equation*}
The function $d_X$ is Lipschitz with Lipschitz constant 1.
By the nerve theorem~\cite{bjorner1995topological}, for any $r>0$ the sublevel set $d_X^{-1}[0,r]$ and the \v{C}ech complex $\check{C}_r(X)$ have isomorphic homology groups with homology groups of $d_X^{-1}[0,r]$ computed via singular homology. 
In particular, the persistence modules of both corresponding filtrations are isomorphic. 
$\conv(X)$ being bounded implies bounded degree-$(n+\delta)$ total persistence of the filtration $(d_X^{-1}[0,r])_r$ for all $\delta>0$, see \Cref{AppendixBoundedTotalPersistence}. 
Then, as in \Cref{PropPolynomialGrowth} there exists a constant $c>0$, such that
\begin{equation*}
n_\omega(A)l_{n+\delta,\omega}(A)^{n+\delta}=\Pers_{n+\delta}(d_X) \leq c \, \amp(d_X)^\delta \frac{n+2\delta}{\delta},
\end{equation*}
with $\amp(d_X)=\max_{x\in \conv(X)}d_X(x)$. 
Note that $\amp(d_X)=d_{\max,\omega}(A)$, so we obtain the desired result for the filtration of \v{C}ech complexes.
\end{proof}

\subsection{Proof of the packing relation}\label{SecProofPackingRelFinal}
Taking expectations, the previous packing lemma yields together with existing limit measures and ergodicity considerations the packing relation, which we can prove in this subsection.
First, we define expected variants of the geometric quantities of the previous subsection.
\vs

\begin{defn}\label{DefGeometricQuantitiesPersistenceExpectation}
Let $\mathfrak{p}$ be a non-zero persistence diagram expectation measure. 
We define the \emph{expected number of persistent homology classes} as
\begin{equation*}
\mathfrak{n}(A):=\int_{\Delta}  \mathfrak{p}(A)(\dd x) = \mathfrak{p}(A)(\Delta).
\end{equation*}
Let $q>0$. The \emph{expected degree-$q$ persistence} is defined as
\begin{equation*}
\mathfrak{l}_q(A):= \left[\frac{1}{\mathfrak{n}(A)} \int_{\Delta} \pers(x)^q \, \mathfrak{p}(A)(\dd x)\right]^{1/q}.
\end{equation*}
With $Y_\omega(A)\subset \Delta$ the multiset of atoms of $\rho_\omega(A)$, we define the \emph{expected maximum death} as
\begin{equation*}
\mathfrak{d}_{\max}(A):=  \lim_{p\to \infty} \left[  \int_{\Delta} d(x)^p\, \mathfrak{p}(A)(\dd x)\right]^{1/p}.
\end{equation*}
\end{defn}
\vs

The following two propositions justify the nomenclature of \Cref{DefGeometricQuantitiesPersistenceExpectation} in light of \Cref{DefGeometricQuantitiesPersistence}.
\vs

\begin{prop}\label{PropExpectationNumberMaxDeath}
Let $\mathfrak{p}$ be an existing non-zero persistence diagram expectation measure and $A\in \mathcal{B}^n_b$. 
Then
\begin{equation*}
\mathfrak{n}(A) = \mathbb{E}[n_\omega(A)].
\end{equation*}
If the measure $\mathfrak{p}(A)$ is boundedly finite, then $\mathfrak{n}(A) < \infty$.
\end{prop}
\vs 

We get similar statements for $\mathfrak{d}_{\max}(A)$ and $\mathfrak{l}_q(A)$ for sufficiently large point clouds.
\vs 

\begin{prop}\label{PropExpectationDegreeQPersistenceLength}
Let $\xi$ be a stationary and ergodic simple point process on $\rr^n$ having all finite moments. 
Let $\mathfrak{p}$ be the persistence diagram expectation measure computed from $\xi$, assumed to exist and to be non-zero. 
Let $\{A_k\}$ be a balanced convex averaging sequence and $\epsilon>0$. 
Then for $k$ sufficiently large:
\begin{equation}\label{EqdmaxAkErgodic}
 |\mathfrak{d}_{\max}(A_k) - \mathbb{E}[d_{\max,\omega}(A_k)]| < \epsilon.
\end{equation}
We find for all $q>0$ and $k$ sufficiently large that
\begin{equation}\label{EqlqAkErgodic}
|\mathfrak{l}_q(A_{k}) - \mathbb{E}[l_{q,\omega}(A_{k})]| < \epsilon.
\end{equation}
For any $p,q\neq 0$ and $\omega\in\Omega$ we a.s. have for sufficiently large $k$:
\begin{equation}\label{EqErgodicConvergenceFracdmaxlq}
    \left| \mathbb{E}\left[\frac{d_{\max,\omega}(A_k)^p}{l_{q,\omega}(A_k)^q} \right] - \lim_{k'\to\infty}\frac{d_{\max,\omega}(A_{k'})^p}{l_{q,\omega}(A_{k'})^q}\right| < \epsilon.
\end{equation}
In particular, this a.s. implies with \Cref{EqdmaxAkErgodic,EqlqAkErgodic} for sufficiently large $k$:
\begin{equation*}
    \left|\mathfrak{d}_{\max}(A_k) -\lim_{k'\to\infty}d_{\max, \omega}(A_{k'})\right| < \epsilon,
\end{equation*}
and for all $q>0$:
\begin{equation*}
    \left|\mathfrak{l}_q(A_k) - \lim_{k'\to\infty} l_{q, \omega}(A_{k'})\right| < \epsilon.
\end{equation*}
Let $A\in \mathcal{B}_b^n$. 
If the measure $\mathfrak{p}(A)$ is boundedly finite and $\supp(\mathfrak{p}(A))\subset \Delta$ is bounded, then  $\mathfrak{d}_{\max}(A)<\infty$ and $\mathfrak{l}_q(A)<\infty$ for all $q>0$.
\end{prop}
\vs

The proofs of Propositions \ref{PropExpectationNumberMaxDeath} and \ref{PropExpectationDegreeQPersistenceLength} are postponed to Appendix \ref{AppendixProofsPropsGeometricQuantities}.
\vs

Self-similar scaling of persistence diagram expectation measures manifests itself in the characteristic scaling behavior of corresponding geometric quantities.
\vs

\begin{lem}\label{LemmaScalingBehavior}
Let $(\mathfrak{p}(t))_{t\in (T_0,T_1)}$ be a family of existing non-zero persistence diagram expectation measures, which scales self-similarly between $T_0$ and $T_1$ with exponents $\eta_1,\eta_2$. 
The $t$-dependence of any geometric quantity constructed from $\mathfrak{p}(t)$ is denoted by an additional $t$-argument. 
Let $\{A_k\}$ be a balanced convex averaging sequence. 
Then for all $t,t'\in (T_0,T_1)$, $q\geq 1$ and $k$ sufficiently large,
\begin{subequations}
\begin{align*}
\mathfrak{n}(t,A_k) =&\,  (t/t')^{ - \eta_2} \mathfrak{n}(t',A_k),\\
\mathfrak{l}_q(t,A_k) =&\, (t/t')^{\eta_1} \mathfrak{l}_q(t',A_k),\\
\mathfrak{d}_{\max}(t,A_k) =&\, (t/t')^{\eta_1} \mathfrak{d}_{\max}(t',A_k).
\end{align*}
\end{subequations}
\end{lem}

\begin{proof}
The derivation of the first two equations follows analogously to the third via push-forward measures and changing integration variables. 
We let $t,t'\in (T_0,T_1)$, set $f_{t,t'}(x):=(t/t')^{\eta_1}x$ for all $x\in \Delta$ and note that $f_{t,t'}(\Delta)=\Delta$. 
Let $k$ be sufficiently large. 
We compute,
\begin{subequations}
\begin{align*}
\mathfrak{d}_{\max}(t,A_k) = &\, \lim_{p\to \infty} \left[  \int_{\Delta} d(x)^p\, \mathfrak{p}(t,A_k)(\dd x)\right]^{1/p}\\
=&\, \lim_{p\to\infty} \left[ (t/t')^{-\eta_2} \int_{\Delta} d(x)^p \, ((f_{t,t'})_*\mathfrak{p}(t',A_k))(\dd x)  \right]^{1/p}\\
=&\, \lim_{p\to\infty} \left[ (t/t')^{-\eta_2} \int_{\Delta} d(f_{t,t'}(x))^p\,  \mathfrak{p}(t',A_k)(\dd x)  \right]^{1/p}\\
= &\,  \lim_{p\to \infty} (t/t')^{\eta_1 - \eta_2/p} \left[  \int_{\Delta} d(x)^p\, \mathfrak{p}(t',A_k)(\dd x)\right]^{1/p}\\
= &\, (t/t')^{\eta_1} \mathfrak{d}_{\max}(t',A_k).
\end{align*}
\end{subequations}
\end{proof}

Finally, based on the devised probabilistic setting and \Cref{LemmaPackingLemma} we can prove the packing relation as stated in \Cref{ThmPackingRelation}.

\begin{proof}[Proof of the packing relation (\Cref{ThmPackingRelation})]
The time-dependence of any geometric quantity constructed from $\rho_\omega (t,\cdot)$ and $\mathfrak{p}(t,\cdot )$ is again denoted by an additional $t$-argument. 
We derive the packing relation from evaluation on a balanced convex averaging sequence $\{A_k\}$. 
From \Cref{LemmaPackingLemma} we obtain that for an arbitrary $\delta>0$,
\begin{equation}\label{EqUpperBound}
n_\omega(t,A_k) \leq \frac{c\, (n+2\delta)}{\delta} \frac{d_{\max,\omega}(t,A_k)^{\delta}}{l_{n+\delta, \omega}(t,A_k)^{n+\delta}}.
\end{equation}
Using \Cref{PropExpectationDegreeQPersistenceLength} we a.s. find for a sample $\omega\in\Omega$, $\epsilon>0$ and $k$ sufficiently large
\begin{equation*}
\epsilon > \left|\mathfrak{d}_{\max}(t,A_k) -\lim_{k'\to\infty}d_{\max, \omega}(t,A_{k'})\right|,
\end{equation*}
and a.s. for $q>0$:
\begin{equation*}
\epsilon> \left|\mathfrak{l}_q(t,A_k) - \lim_{k'\to\infty} l_{q, \omega}(t,A_{k'})\right|.
\end{equation*}
Exploiting \Cref{EqErgodicConvergenceFracdmaxlq} from \Cref{PropExpectationDegreeQPersistenceLength}, \Cref{EqUpperBound} a.s. yields for sufficiently large $k$ upon taking expectations
\begin{subequations}
\begin{align*}
\mathfrak{n}(t,A_k) \leq &\,\lim_{k'\to \infty} \frac{c\, (n+2\delta)}{\delta} \frac{d_{\max, \omega}(t,A_{k'})^{\delta}}{l_{n+\delta, \omega}(t,A_{k'})^{n+\delta}} + O(\epsilon^\delta)\\
=&\,\frac{c\, (n+2\delta)}{\delta} \frac{\mathfrak{d}_{\max}(t,A_k)^{\delta}}{\mathfrak{l}_{n+\delta}(t,A_k)^{n+\delta}} + O(\epsilon^\delta).
\end{align*}
\end{subequations}
Exploiting self-similarity and \Cref{LemmaScalingBehavior}, we find for any $t,t'\in(T_0,T_1)$,
\begin{equation*}
\frac{\mathfrak{d}_{\max}(t,A_k)^{\delta}}{\mathfrak{l}_{n+\delta}(t,A_k)^{n+\delta}} = (t/t')^{-n\eta_1}  \frac{\mathfrak{d}_{\max}(t',A_k)^{\delta}}{\mathfrak{l}_{n+\delta}(t',A_k)^{n+\delta}}.
\end{equation*}
Hence,
\begin{equation}\label{EqIneqPackingRelationProof}
\mathfrak{n}(t,A_k) = (t/t')^{-\eta_2} \mathfrak{n}(t',A_k) \leq (t/t')^{-n\eta_1}\frac{c\, (n+2\delta)}{\delta} \frac{\mathfrak{d}_{\max}(t',A_k)^{\delta}}{\mathfrak{l}_{n+\delta}(t',A_k)^{n+\delta}} + O(\epsilon^\delta).
\end{equation}
We assume that $\eta_2\neq n\eta_1$ and that $t,t'\in (T_0,T_1)$ exist with
\begin{equation}\label{EqSuffExtended}
t/t' > \frac{c(n+2\delta)}{\delta} \max\left\{ \frac{\mathfrak{d}_{\max}(t,A_k)^\delta}{\mathfrak{n}(t,A_k) \mathfrak{l}_{n+\delta}(t,A_k)^{n+\delta}},  \frac{\mathfrak{d}_{\max}(t',A_k)^\delta}{\mathfrak{n}(t',A_k) \mathfrak{l}_{n+\delta}(t',A_k)^{n+\delta}}\right\}^{1/|n\eta_1 - \eta_2|},
\end{equation}
for any $k$ sufficiently large. 
Then, either for $n\eta_1<\eta_2$,
\begin{equation*}
(t/t')^{\eta_2-n\eta_1} = (t'/t)^{n\eta_1-\eta_2} >  \frac{c(n+2\delta)}{\delta} \frac{\mathfrak{d}_{\max}(t,A_k)^\delta}{\mathfrak{n}(t,A_k) \mathfrak{l}_{n+\delta}(t,A_k)^{n+\delta}},
\end{equation*}
or for $n\eta_1 > \eta_2$,
\begin{equation*}
(t/t')^{n\eta_1-\eta_2} >  \frac{c(n+2\delta)}{\delta} \frac{\mathfrak{d}_{\max}(t',A_k)^\delta}{\mathfrak{n}(t',A_k) \mathfrak{l}_{n+\delta}(t',A_k)^{n+\delta}},
\end{equation*}
both of them being in contradiction to \Cref{EqIneqPackingRelationProof} for sufficiently small $\epsilon$. 
Thus, the desired equality $\eta_2=n\eta_1$ a.s. follows, provided that the interval $(T_0,T_1)$ is sufficiently extended in the sense of \Cref{EqSuffExtended}.
\end{proof}

\section{Examples}\label{SecExamples}
This section is devoted to three examples for self-similar scaling. 
In \Cref{SecExample1} we consider a Poisson point process with time-dependent intensity and show self-similar scaling of corresponding persistence diagram expectation measures. 
The second example, given in \Cref{SecExample2}, considers persistence diagram measures of point clouds sampled uniformly from sublevel sets of smooth functions which themselves scale self-similarly. 
\Cref{SecExampleQuantumPhysics} describes an application of the deduced results in quantum many-body physics~\cite{spitz2020finding}.

\subsection{Poisson point process with power-law scaling intensity}\label{SecExample1}
We consider time-dependent Poisson point processes in the following sense, which generalize usual Poisson point processes~\cite{last2017lectures}.
\vs

\begin{defn}\label{DefTimeDepPoisson}
A family of point processes $(\xi(t))_{t\in[1,\infty)}$ on $\rr^n$ is a \emph{time-dependent Poisson point process on $\rr^n$}, if there exists an intensity function $\gamma:[1,\infty)\to(0,\infty)$, such that for each $t\in [1,\infty)$,
\begin{enumerate}[(i)]
\item the expected number of points in $A\in \mathcal{B}^n$ is $\mathbb{E}[\xi(t,A)]=\gamma(t) \lambda_n(A)$,
\item for every $A\in \mathcal{B}^n$ the distribution of $\xi(t,A)$ is $\mathbb{P}[\xi(t,A)=k] = \mathrm{Po}(\gamma(t)\lambda_n(A);k)$ for all $k\in\nn$, $\mathrm{Po}(\lambda;k) = (\lambda^k/k!)\exp(-\lambda)$ denoting the Poisson distribution,
\item for every $m\in \nn$ and all pairwise disjoint Borel sets $A_1,\dots,A_m\in \mathcal{B}^n$ the random variables $\xi(t,A_1),\dots,\xi(t,A_m)$ are independent.
\end{enumerate}
\end{defn}
\vs

A time-dependent Poisson process with properties (i) to (iii) defines a Poisson point process at each time $t$, individually. 
It is a basic result from the theory of point processes that such a point process $\xi(t)$ is stationary and ergodic having all finite moments for each $t\in [1,\infty)$~\cite{daley2003introduction, daley2007, last2017lectures}.
\vs

\begin{prop}\label{PropTimeDepPoisson}
Let $(\xi(t))_{t\in[1,\infty)}$ be a time-dependent Poisson point process on $\rr^n$ with intensity function
\begin{equation*}
\gamma(t) = \gamma_0 t^{-n\eta_1},
\end{equation*}
$\gamma_0>0$ and $\eta_1\geq 0$. 
Let $\{A_k\}$ be a balanced convex averaging sequence. 
Then the family $(\mathfrak{p}(t,A_k) / \lambda_n(A_k))_t$ of persistence diagram expectation measures computed from $(\xi(t))_t$ normalized to the volume of the convex sets converges vaguely for $k\to\infty$ to a family of Radon measures $(\mathfrak{P}(t))$ which scales self-similarly between $1$ and $\infty$ with exponents $\eta_1$ and $n\eta_1$.
\end{prop}

\begin{proof}
Let $\tilde{\xi}_0$ be the Poisson point process with intensity $\gamma_0$, such that $\mathbb{E}[\tilde{\xi}_0(A)]=\gamma_0\lambda_n(A)$ for $A\in \mathcal{B}^n_b$. 
We draw a sample point cloud $X_{\tilde{\xi}_{0,\omega}}(A)$ of $\tilde{\xi}_0$ and define for all $t\in [1,\infty)$
\begin{equation*}
X_{\tilde{\xi}_\omega}(t,A):= \big(t^{\eta_1} X_{\tilde{\xi}_{0,\omega}}(A)\big)\cap A.
\end{equation*}
This defines a point process $\tilde{\xi}(t)$ with $\# X_{\tilde{\xi}_\omega}(t,A) = \tilde{\xi}_{0,\omega}(t^{-\eta_1}A)$ for each $t$. 
Its intensity can be computed,
\begin{equation*}
\mathbb{E}[\tilde{\xi}(t,A)] = \mathbb{E}[\# ((t^{\eta_1} X_{\tilde{\xi}_{0,\omega}}(A))\cap A)] = \mathbb{E}[\# X_{\tilde{\xi}_{0,\omega}}(t^{-\eta_1}A)] = \gamma_0 t^{-n\eta_1} \lambda_n(A),
\end{equation*}
that is, the intensities of $\xi(t)$ and $\tilde{\xi}(t)$ agree for all $t$. 
Also, we have 
\begin{equation*}
\mathbb{P}[\tilde{\xi}(t,A)=k] = \mathbb{P}[\tilde{\xi}_0(t^{-\eta_1}A)=k] = \mathrm{Po}(\gamma_0 \lambda_n(t^{-\eta_1}A);k) = \mathrm{Po}(\gamma(t)\lambda_n(A);k)\,.
\end{equation*}
Furthermore, let $A_1,A_2\subset \rr^n$ be disjoint Borel sets. 
Then, $\tilde{\xi}_0(A_1)$ and $\tilde{\xi}_0(A_2)$ are independent random variables, since $\tilde{\xi}_0$ is a Poisson point process. 
Let $x_i\in A_i$. 
Then, $|x_1-x_2|>0$ and $t^{-\eta_1}|x_1-x_2|>0$, i.e., $t^{-\eta_1}A_1$ and $t^{-\eta_1}A_2$ are disjoint, too. 
Thus, $\tilde{\xi}(t,A_1)$ and $\tilde{\xi}(t,A_2)$ are also independent random variables. 
The Poisson point process being uniquely characterized by the properties (i) to (iii) of \Cref{DefTimeDepPoisson}, $\xi(t)$ and $\tilde{\xi}(t)$ can be identified for all $t$.

We employ the strong law of large numbers for persistent Betti numbers (\Cref{CorStrongLawPersBetti}) and denote the limiting persistent Betti numbers at time $t$ as $\hat{\beta}^{r,s}_\ell(t)$, such that for $k\to\infty$:
\begin{equation*}
\frac{\mathbb{E}[\beta^{r,s}_\ell (\mathcal{C}(X_{\xi}(t,A_k)))]}{\lambda_n(A_k)}\to \hat{\beta}^{r,s}_\ell (t).
\end{equation*}
We compute,
\begin{subequations}
\begin{align*}
\frac{\mathbb{E}[\beta^{r,s}_\ell (\mathcal{C}(X_{\xi}(t,A_k)))]}{\lambda_n(A_k)} =&\; \frac{\mathbb{E}[\beta^{r,s}_\ell (\mathcal{C}(t^{\eta_1}(X_{\tilde{\xi}_0}(A_k)\cap t^{-\eta_1}A_k)))]}{\lambda_n(A_k)}\\
=&\; \frac{\mathbb{E}[\beta_\ell^{t^{-\eta_1}r,t^{-\eta_1}s}(\mathcal{C}(X_{\tilde{\xi}_0}(A_k)\cap t^{-\eta_1}A_k))]}{\lambda_n(A_k)}\\
=&\; \frac{t^{-n\eta_1} \mathbb{E}[\beta_\ell^{t^{-\eta_1}r,t^{-\eta_1}s}(\mathcal{C}(X_{\tilde{\xi}_0}(t^{-\eta_1}A_k)))]}{\lambda_n(t^{-\eta_1}A_k)}\overset{k\to\infty}{\longrightarrow} t^{-n\eta_1} \hat{\beta}_\ell^{t^{-\eta_1}r,t^{-\eta_1}s}(1).
\end{align*}
\end{subequations}
Both limits need to agree, i.e.,
\begin{equation}\label{EqPersBettiNumbersScalingPoisson}
\hat{\beta}^{r,s}_\ell (t) =  t^{-n\eta_1} \hat{\beta}_\ell^{t^{-\eta_1}r,t^{-\eta_1}s}(1).
\end{equation}

With $\mathfrak{p}(t,A_k)$ the persistence diagram expectation measure computed from $\xi(t)$ and evaluated on $A_k$, by means of \Cref{ThmExistenceLimitingRadonMeasure} we obtain the existence of a unique Radon measure $\mathfrak{P}(t)$ at each time $t$, such that for $k\to\infty$
\begin{equation*}
\frac{\mathfrak{p}(t,A_k)}{\lambda_n(A_k)}\overset{v}{\longrightarrow} \mathfrak{P}(t).
\end{equation*}
Starting from \Cref{EqPersBettiNumbersScalingPoisson}, identical arguments which lead to the proof of Theorem 1.5 in \cite{hiraoka2018} here lead to
\begin{equation*}
\mathfrak{P}(t)(B) = t^{-n\eta_1} \mathfrak{P}(1)(t^{-\eta_1}B)
\end{equation*}
for any $B\in\mathcal{B}(\Delta)$.
\end{proof}

We note that \Cref{EqPersBettiNumbersScalingPoisson} is the usual scaling property of persistent Betti numbers for Poisson point processes \cite{goel2019strong}.

\subsection{Scaling from function sublevel sets}\label{SecScalingFunctionSublevels}\label{SecExample2}
The following example establishes a relation between self-similar scaling exponents derived from moments of random functions and those derived from persistence diagram measures.
This is based on the assumption of an underlying rescaling of the metric.

\paragraph{Scaling functions.}
Let $(T_0,T_1)\subseteq \rr$ be a time interval. 
Let $(\Omega, \mathcal{E}, \mathbb{P})$ be a probability space with $\Omega$ a set of smooth functions $Y:(T_0,T_1)\times \rr^n \to \rr$, $\mathcal{E}$ some event space, such that for arbitrary $t,t'\in (T_0,T_1), x\in \rr^n$:
\begin{equation}\label{EqAlmostSureRescalingSamples}
    \mathbb{P}\left[Y(t,x)=\left(\frac{t}{t'}\right)^\alpha Y(t',(t/t')^\beta x)\right] = 1,
\end{equation}
i.e., a.s. all samples reveal self-similar scaling in time. 
Then all moments of $\mathbb{P}$ show self-similar behavior as well:
\begin{align*}
    \mathbb{E}\left[\prod_{i=1}^N Y(t,x_i)\right] = \int \prod_{i=1}^N Y(t,x_i) \mathbb{P}(\dd Y) =&  \left(\frac{t}{t'}\right)^{N\alpha} \int \prod_{i=1}^N Y(t',(t/t')^\beta x_i) \mathbb{P}(\dd Y)\\
    =& \left(\frac{t}{t'}\right)^{N\alpha} \mathbb{E}\left[ \prod_{i=1}^N Y(t',(t/t')^\beta x_i)\right],
\end{align*}
where $x_i\in\rr^n$ for $i=1,\dots,N$, $N\in\nn$. 
Let $f:\rr\to \rr$ be a homogeneous function of degree $\kappa>0$ with $f(x)\geq 0$ for all $x\in\rr$. 
We define submanifolds of $\rr^n$ via
\begin{equation*}
    X_\nu[Y](t):=\{x\in \rr^n\, |\, f(Y(t,x))\leq \nu\}.
\end{equation*}
Then, a.s. we find by \Cref{EqAlmostSureRescalingSamples}
\begin{equation}\label{EqScalingSublevels}
    X_\nu[Y](t) = \left\{ x\in \rr^n \, \left|\, f(Y(t',(t/t')^\beta x))\leq (t/t')^{-\kappa \alpha} \nu\right.\right\} = \left(\frac{t}{t'}\right)^{-\beta} X_{(t/t')^{-\kappa \alpha} \nu} [Y](t').
\end{equation}

\paragraph{Scaling of the persistence diagram measure.}
We demonstrate in the following that this likely implies scaling of persistence diagram measures computed from point clouds appropriately sampled from $X_\nu[Y](t)$.

Let $\{A_k\}$ be a balanced convex averaging sequence in $\rr^n$, $\varepsilon>0$. 
We choose $q$ points $x_i\in X_\nu[Y](t)\cap A_k$ such that $\bigcup_{i=1}^q B_\varepsilon(x_i)$ is a cover of $X_\nu[Y](t)\cap A_k$. 
From $X_\nu[Y](t)\cap A_k$ we sample a point cloud $X_{\nu,M}[Y](t,A_k)$ consisting of a number $M$ of i.i.d. uniformly distributed points. 
Then there exists $\alpha>0$ such that 
\begin{equation*}
\mathbb{E}[\#(X_{\nu,M}[Y](t,A_k) \cap B_{\varepsilon}(x_i))]>\alpha
\end{equation*}
for all $i$. 

With quantifiably high confidence, homology can be inferred from random point samples of manifolds as in~\cite{niyogi2008finding}, which includes a range of probability bounds on point sampling.
In particular, by Lemma 5.1 of~\cite{niyogi2008finding} we then have for $\delta>0$ that with probability $1-\delta$
\begin{equation*}
    X_{\nu,M}[Y](t,A_k) \cap B_{\varepsilon}(x_i)\neq \emptyset
\end{equation*}
for all $i$, given that $M> (\log q - \log \delta)/\alpha$. 
Thus, with probability $1-\delta$ we find
\begin{align*}
    X_{\nu,M}[Y](t,A_k)\subseteq&\; X_\nu[Y](t)\cap A_k \subseteq \bigcup_{x\in X_{\nu,M}[Y](t,A_k)} B_{2\varepsilon}(x) \subseteq\; (X_\nu[Y](t)\cap A_k )\oplus B_{2\varepsilon}(0),
\end{align*}
$\oplus$ denoting the Minkowski sum among subsets of $\rr^n$.
To this end, we have shown that with probability $1-\delta$ a $2\varepsilon$-interleaving between point clouds $X_{\nu,M}[Y](t,A_k)$ and sublevels sets $X_\nu[Y](t)\cap A_k$ exists.

For simplicity we denote the filtration of \v{C}ech complexes of $X_{\nu,M}[Y](t,A_k)$ by $\mathcal{C}_t$ and the filtration of growing balls around $X_\nu[Y](t)\cap A_k$ by $\mathcal{D}_t$, i.e., $( (X_\nu[Y](t)\cap A_k) \oplus B_r(0))_r$. 
We denote the persistence module computed via simplicial homology from $\mathcal{C}_t$ by $H_*(\mathcal{C}_t)$ and the persistence module computed via singular homology from $\mathcal{D}_t$ by $H_*(\mathcal{D}_t)$. 
Given the $2\varepsilon$-interleaving, which exists with probability $1-\delta$, we then find an upper bound for their interleaving distance,
\begin{align*}
    d_I(H_*(\mathcal{C}_t),H_*(\mathcal{D}_t)) =&\; \inf \{ \epsilon \, |\, \textrm{there exists an }\epsilon\textrm{-interleaving between }H_*(\mathcal{C}_t),H_*(\mathcal{D}_t)\}\\
    \leq&\; 2\varepsilon.
\end{align*}
Following Theorem 3.5 of~\cite{bauer2013induced} we find equality between the bottleneck and the interleaving distance in our case:
\begin{align*}
    & d_B(H_*(\mathcal{C}_t),H_*(\mathcal{D}_t))\\
    & = \inf \{ \epsilon\in[0,\infty) \, |\, \textrm{there exists an }\epsilon\textrm{-matching between }H_*(\mathcal{C}_t),H_*(\mathcal{D}_t)\} \\ 
    & =d_I(H_*(\mathcal{C}_t),H_*(\mathcal{D}_t))\leq 2\varepsilon,
\end{align*}
with an $\epsilon$-matching defined as in~\cite{bauer2013induced}.
This implies that if $B\in\mathcal{B}(\Delta)$ fulfills $(b,d)\in B \,\Longrightarrow\, d-b>4\varepsilon$, then the persistence diagram measures are related as follows:
\begin{equation}\label{EqEstimatingPointCloudPersDiag}    \rho(\mathcal{C}_t)(B)\leq \rho(\mathcal{D}_t)(B\oplus B_{2\varepsilon}(0))\quad \textrm{and} \quad \rho(\mathcal{D}_t)(B)\leq \rho(\mathcal{C}_t)(B\oplus B_{2\varepsilon}(0)),
\end{equation}
where the dependence on the sample function $Y$ is implicit here. 
Thus, we find with probability $1-\delta$
\begin{equation}\label{EqDCDInequality}
    \rho(\mathcal{D}_t)(B)\leq \rho(\mathcal{C}_t)(B\oplus B_{2\varepsilon}(0))\leq \rho(\mathcal{D}_t)(B\oplus B_{4\varepsilon}(0))
\end{equation}
and can likely estimate the persistence diagram measure computed from sampled point clouds, $\rho(\mathcal{C}_t)$, from the one of sublevels, directly, $\rho(\mathcal{D}_t)$. 
Employing \Cref{EqScalingSublevels}, we find
\begin{equation*}
    X_\nu[Y](t)\cap A_k = \left(\frac{t}{t'}\right)^{-\beta} (X_{(t/t')^{-\kappa\alpha}\nu}[Y](t')\cap (t/t')^\beta A_k).
\end{equation*}
From this we find the time-dependence of the singular homology groups,
\begin{equation}\label{EqScalingHomologyGroups}
    H_*((X_\nu[Y](t)\cap A_k )\oplus  B_r(0)) = H_*((X_{(t/t')^{-\kappa\alpha}\nu}[Y](t')\cap (t/t')^\beta A_k ) \oplus  B_{(t/t')^\beta r}(0)).
\end{equation}
Hence, we have for any $B\in\mathcal{B}(\Delta)$
\begin{equation}
    \rho(\mathcal{D}_t)(B) = \rho(\mathcal{D}_{t,t'}')((t/t')^\beta B),
\end{equation}
with $\mathcal{D}'_{t,t'}$ the filtration $(( X_{(t/t')^{-\kappa\alpha}\nu}[Y](t')\cap (t/t')^\beta A_k ) \oplus  B_r(0))_r$ of topological spaces. Insertion into \Cref{EqEstimatingPointCloudPersDiag} leads with probability $1-\delta$ to
\begin{equation}\label{EqDCDttprimeInequality}
    \rho(\mathcal{D}_{t,t'}')((t/t')^\beta B) \leq \rho(\mathcal{C}_t)(B\oplus B_{2\varepsilon}(0)) \leq \rho(\mathcal{D}_{t,t'}')((t/t')^\beta B \oplus  B_{4(t/t')^\beta \varepsilon}(0)).
\end{equation}
For any given $\varepsilon$ we can always choose the number of sampled points $M$ sufficiently large for the described constructions to apply. 
To this end, the limit $\varepsilon\to 0$ may be taken in \Cref{EqDCDInequality,EqDCDttprimeInequality}, leading with probability $1-\delta$ for any $B\in\mathcal{B}(\Delta)$ to
\begin{equation}\label{EqScalingrhoCt}
    \lim_{\varepsilon\to 0}\rho(\mathcal{C}_t)(B) = \rho(\mathcal{D}_{t,t'}')((t/t')^\beta B) = \lim_{\varepsilon\to 0}\rho(\mathcal{C}_{t,t'}')((t/t')^\beta B).
\end{equation}
Here, $\mathcal{C}_{t,t'}'$ denotes the filtration of \v{C}ech complexes of $X_{(t/t')^{-\kappa\alpha}\nu, M}[Y](t,(t/t')^\beta A_k)$.

\paragraph{Self-similar scaling of the persistence diagram expectation measure.}
The scaling of sublevel set homology groups as in \Cref{EqScalingHomologyGroups} still encodes temporally varying balanced convex averaging sequences $(t/t')^\beta A_k$. 
We show that for particular sets of functions $\Omega$ the sampling of point clouds makes up a stationary point process, leading via \Cref{ThmExistenceLimitingRadonMeasure} to an overall scaling factor of the persistence diagram expectation measure in the large-volume limit. 
This results in its self-similar scaling. 
As a byproduct the packing relation is independently from the proof of \Cref{ThmPackingRelation} demonstrated for this particular example.

We assume that the samples $f\circ Y$ resemble stationarity in probability for each time individually, i.e., for all $-\infty \leq a\leq b \leq \infty$, $x,y\in \rr^n$ and $t\in (T_0,T_1)$,
\begin{equation}\label{EqAssumptionStationarityProbability}
    \mathbb{P}[a\leq f(Y(t,x)) \leq b] = \mathbb{P}[a\leq f(Y(t,y)) \leq b].
\end{equation}
Then, any point $x\in A_k$ is equally likely included in a point cloud $X_{\nu,M}[Y](t,A_k)$, given that $X_{\nu,M}[Y](t,A_k)$ consists of i.i.d. uniformly distributed points in $X_\nu[Y](t)\cap A_k$. To this end, under assumption (\ref{EqAssumptionStationarityProbability}) $Y(t,\cdot) \mapsto X_{\nu,M}[Y](t,A_k)$ defines a stationary point process on $A_k$ with $Y\in \Omega$. 
Furthermore, the defined point process has all finite moments, again exploiting that $X_{\nu,M}[Y](t,A_k)$ consists of i.i.d. uniformly distributed points.

Thus, we may apply \Cref{ThmExistenceLimitingRadonMeasure}, which yields the existence of limiting volume-averaged persistence diagram expectation measures. 
Let $\mathfrak{p}(\mathcal{C}_{t,t'}'):= \mathbb{E}[\rho(\mathcal{C}_{t,t'}')]$ be the persistence diagram expectation measure computed from the persistence diagram measures of the \v{C}ech complexes of point clouds sampled from sublevel sets with, as a special case, $\mathfrak{p}(\mathcal{C}_t):= \mathbb{E}[\rho(\mathcal{C}_t)]$. 
Implicitly, it depends on the parameter $\varepsilon$. 
By \Cref{ThmExistenceLimitingRadonMeasure} we obtain that a limiting Radon measure $\mathfrak{P}_t$ on $\Delta$ exists, such that for any continuous function $g$ on $\Delta$ with compact support:
\begin{equation}\label{EqWeakLimitTimetRadonMeasure}
    \lim_{k\to\infty} \frac{1}{\lambda_n(A_k)} \lim_{\varepsilon\to 0}\int_\Delta g\,  \mathfrak{p}(\mathcal{C}_t) = \int_\Delta g\,  \mathfrak{P}_t.
\end{equation}
Regard here that $\mathfrak{p}(\mathcal{C}_t)$ implicitly depends on the $A_k$. 
More generally, we find that a Radon measure $\mathfrak{P}_{t,t'}$ exists, which satisfies again for all continuous functions $g$ on $\Delta$ with compact support:
\begin{align*}
     & \left(\frac{t}{t'}\right)^{n\beta}\int_\Delta g((t/t')^{-\beta} x)\,  \mathfrak{P}_{t,t'}(\dd x) \\
     &\qquad =\; \left(\frac{t}{t'}\right)^{n\beta}\lim_{k\to\infty} \frac{1}{\lambda_n((t/t')^\beta A_k)} \lim_{\varepsilon\to 0}\int_\Delta g((t/t')^{-\beta} x)\,  \mathfrak{p}(\mathcal{C}_{t,t'}')(\dd x)\\
     &\qquad =\;\lim_{k\to\infty} \frac{1}{\lambda_n(A_k)} \lim_{\varepsilon\to 0}\int_\Delta g((t/t')^{-\beta}x)\,  \mathfrak{p}(\mathcal{C}_{t,t'}')(\dd x).
\end{align*}
We employed here \Cref{ThmExistenceLimitingRadonMeasure} together with $\{(t/t')^\beta A_k\}$ being a balanced convex averaging sequence. 
As a special case we have $\mathfrak{P}_t = \mathfrak{P}_{t,t}$.
Using \Cref{EqScalingrhoCt}, we arrive at
\begin{equation*}
    \lim_{k\to\infty} \frac{1}{\lambda_n(A_k)} \lim_{\varepsilon\to 0}\int_\Delta g(x)\,  \mathfrak{p}(\mathcal{C}_t)(\dd x) = \lim_{k\to\infty} \frac{1}{\lambda_n(A_k)} \lim_{\varepsilon\to 0}\int_\Delta g((t/t')^{-\beta}x)\,  \mathfrak{p}(\mathcal{C}_{t,t'}')(\dd x),
\end{equation*}
which leads with \Cref{EqWeakLimitTimetRadonMeasure} to
\begin{equation}\label{EqLargeVolumeAverageScaling}
    \int_\Delta g(x) \,\mathfrak{P}_t (\dd x) = \left(\frac{t}{t'}\right)^{n\beta} \int_\Delta g((t/t')^{-\beta}x)\, \mathfrak{P}_{t,t'}(\dd x).
\end{equation}
If we assume that in a sufficiently large regime of filtration parameters $\nu$ under consideration persistence diagram expectation measures are independent from $\nu$, then $\mathfrak{P}_{t,t'}$ is independent from $t'$ up to the power-law denoted in the argument of $g$ on the right-hand side of \Cref{EqLargeVolumeAverageScaling}.
In the limit of large volumes we obtain self-similar scaling of limiting volume-averaged persistence diagram expectation measures, including an independent derivation of the packing relation for the considered example.

\subsection{Application: Numerical simulations in non-equilibrium quantum physics}\label{SecExampleQuantumPhysics}
Studying quantum physics far from equilibrium, indications for self-similar scaling of persistence diagram expectation measures have been found in numerical simulations of the non-relativistic Bose gas by the authors of the present manuscript and others~\cite{spitz2020finding}. 
The scaling behavior can be attributed to the existence of a so-called non-thermal fixed point in the dynamics of the system, characterized by such scaling and typically verified by investigating the behavior of correlation functions. 
The study demonstrated for the first time, that persistent homology can be used to detect dynamical quantum phenomena.

The simulations have been carried out in the classical-statistical regime of many particles interacting comparably weakly. 
In this regime the quantum physics can be accurately mapped to a Gaussian ensemble of complex-valued fields at initial time $t=0$, $\psi_\omega(0):\Lambda \to \cc$, $\omega\in \Omega$, $\Lambda\subset \rr^2$ a uniform finite square lattice with constant lattice spacing, which are then time-evolved individually according to the so-called Gross-Pitaevskii differential equation in order to obtain $\psi_\omega(t)$ from $\psi_\omega(0)$ for all $t\in [0,\infty)$. 
This has yielded an ensemble of fields $(\psi_\omega)_{\omega\in\Omega}$, $\psi_\omega:[0,\infty)\times \Lambda\to\cc$. 
Theoretical predictions have been computed as expectation values with respect to this ensemble of fields.

Given a sample $\psi_\omega$, point clouds have been constructed as subsets of the lattice $\Lambda$ at individual times for filtration parameters $\bar{\nu} \in [0,\infty)$,
\begin{equation*}
X_{\bar{\nu},\omega}(t):= |\psi_\omega(t)|^{-1} [0,\bar{\nu}]\subset \rr^2.
\end{equation*}
$\Lambda$ being finite, the $X_{\bar{\nu},\omega}(t)$ are finite sets, too. 
Their alpha complexes and persistence diagrams have been computed. 
For an impression of alpha complexes of different radii see Fig. \ref{FigNonequQuantumPhysicsAlphaComplexes}. 
Finally, expectations for functional summaries such as smooth variants of the distribution of birth radii, i.e., $\mathfrak{p}(t,\Lambda)([r_1,r_2]\times[0,\infty))$, and $\mathfrak{d}_{\max}(t,\Lambda)$ have been computed.%
\footnote{Numerically, convolutions with a Gaussian mollifier of arbitrary width, which are necessary to obtain smoothed distributions of death radii, for instance, are irrelevant. 
Simple binning procedures for such distributions and averaging over a sufficiently large number of samples allow for a proper computation of expectation values.}

\begin{figure}
\begin{center}
\includegraphics[scale=0.7]{{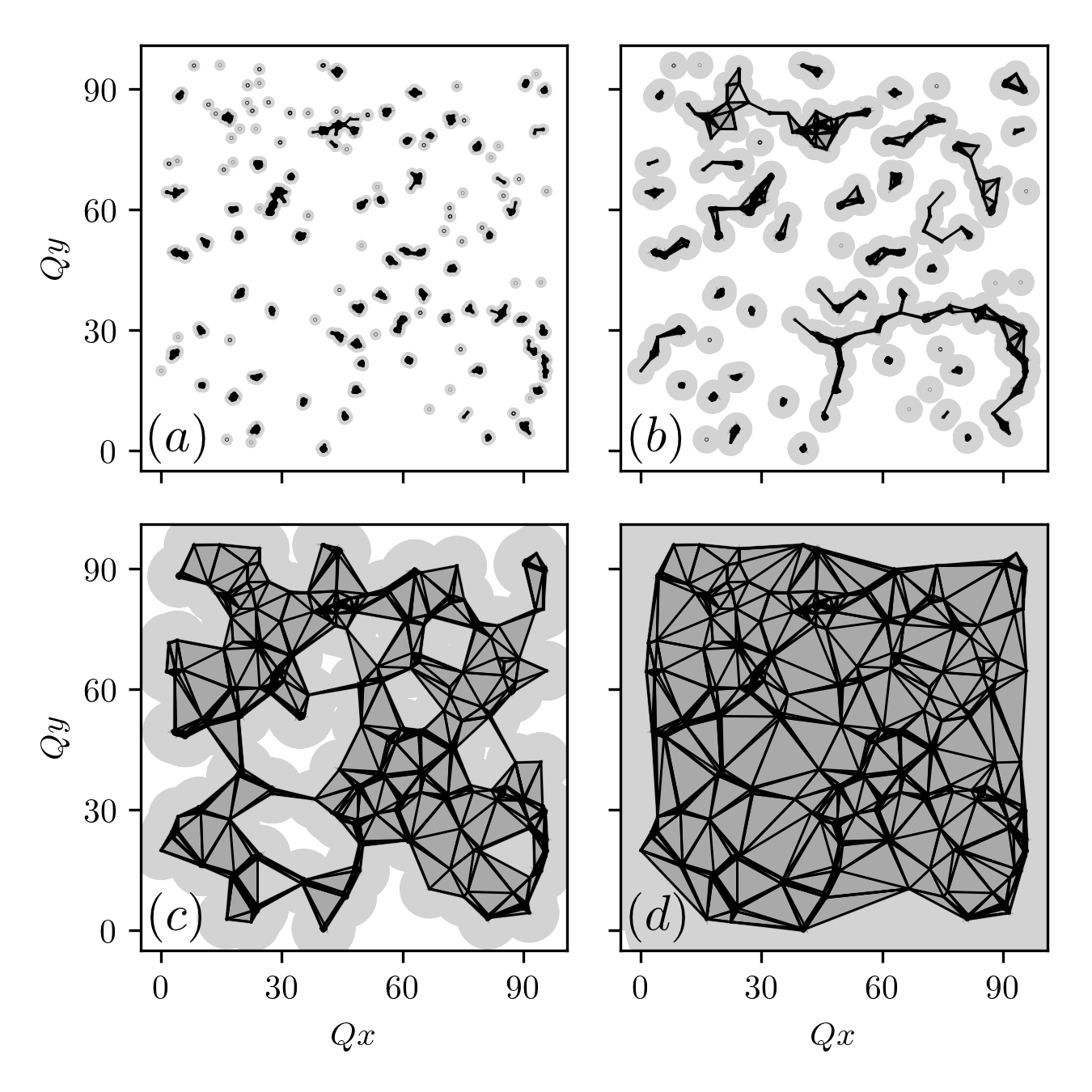}}
\caption{Alpha complexes of increasing radii from (a) to (d) at a non-thermal fixed point, reprinted from~\cite{spitz2020finding}.}\label{FigNonequQuantumPhysicsAlphaComplexes}
\end{center}
\end{figure}

\begin{figure}
\begin{center}
\includegraphics[scale=0.85]{{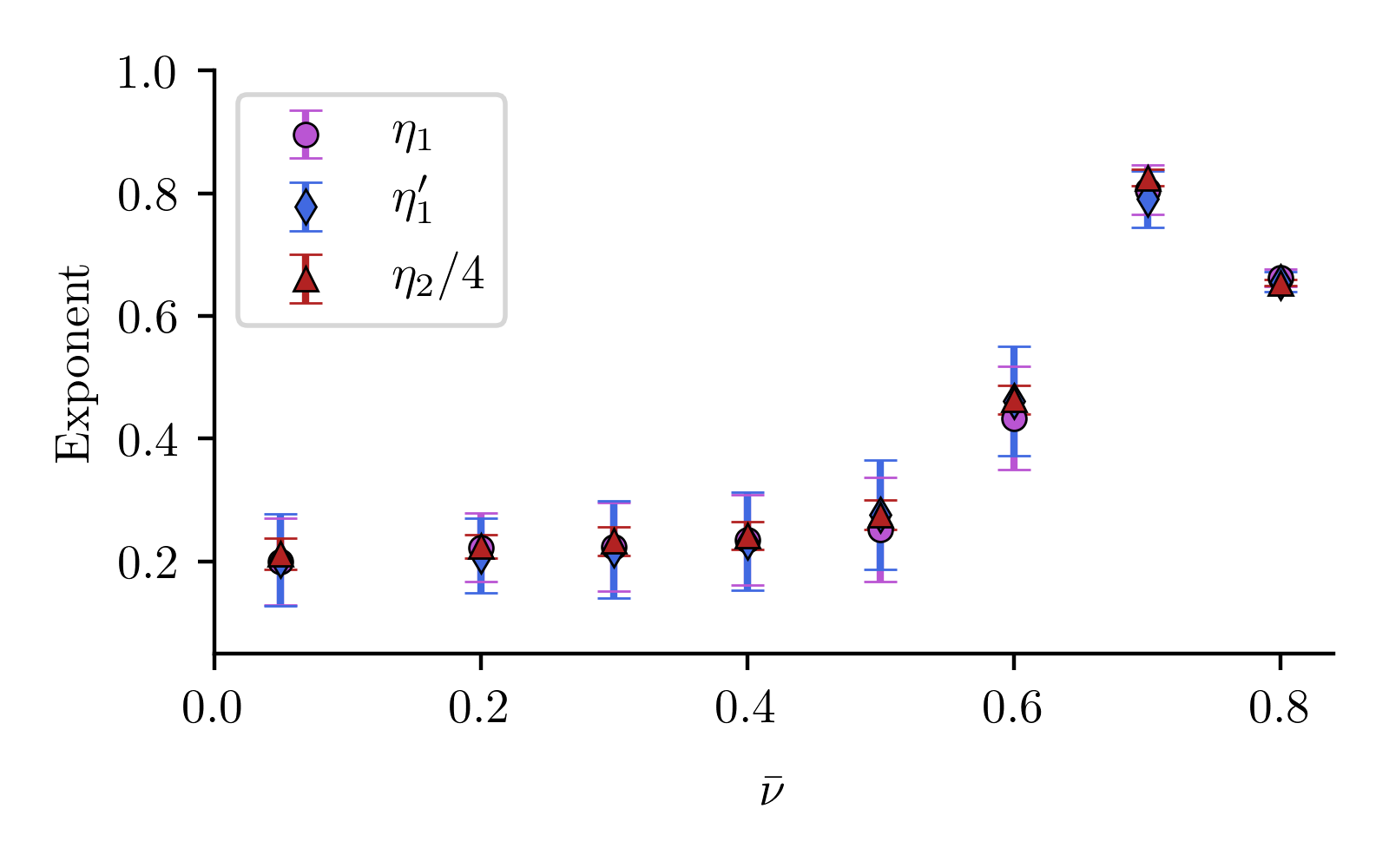}}
\caption{Self-similar scaling exponents at a non-thermal fixed point, reprinted from~\cite{spitz2020finding}. $\eta_1'$ has been introduced here as an independent scaling exponent for death radii in \Cref{EqScalingAnsatzLebesgueDensity}.}\label{FigNonequQuantumPhysicsExponents}
\end{center}
\end{figure}

A scaling ansatz has been made for the Lebesgue density of the persistence diagram expectation measure,
\begin{equation}\label{EqScalingAnsatzLebesgueDensity}
\tilde{\mathfrak{p}}(t,\Lambda)(b,d) = (t/t')^{-\eta_2} \tilde{\mathfrak{p}}(t',\Lambda)((t/t')^{-\eta_1}b, (t/t')^{-\eta_1}d),
\end{equation}
resulting in the self-similar scaling of persistent homology quantities. 
A possible singular contribution to the persistence diagram expectation measure as it appears in \Cref{PropExistenceLebesgueDensities} has been ignored. 
Resulting exponents are summarized in \Cref{FigNonequQuantumPhysicsExponents}. 
By simple application of the transformation theorem we find
\begin{equation*}
\mathfrak{n}(t,A) = (t/t')^{2\eta_1 - \eta_2} \mathfrak{n}(t',A).
\end{equation*}
Comparing with a family of persistence diagram expectation measures that scales self-similarly with exponents $\tilde{\eta}_1,\tilde{\eta}_2$ but yields the same scaling behavior of geometric quantities, we obtain $\eta_1=\tilde{\eta}_1$ and via \Cref{LemmaScalingBehavior}
\begin{equation*}
\eta_2 = \tilde{\eta}_2 + 2\tilde{\eta}_1. 
\end{equation*}
The packing relation given in \Cref{ThmPackingRelation} then translates for the Lebesgue density scaling exponents to $\eta_2 = (2+n)\eta_1$, in accordance with \Cref{FigNonequQuantumPhysicsExponents}.

Numerically, the underlying mathematical assumptions of this work can be shown, in particular stationarity and ergodicity of the corresponding point process. 
It has been checked that varying the distance between neighboring lattice sites does not affect the results, which is theoretically motivated by the stability theorems for persistent homology~\cite{Cohen-Steiner2007,Cohen-Steiner2010}. 
To this end, the specific microscopic lattice geometry can be regarded as unimportant for the displayed results. 
Furthermore, for the chosen initial ensemble of fields, $\psi_\omega(0)$, physical quantities numerically have revealed (approximate) translation-invariance across the lattice, reflecting stationarity. 
It is easy to verify numerically that classical-statistical simulations are ergodic in the sense of recovering expectation values for intensive quantities in the limit of large lattices (in the field of quantum dynamics this property is better known as self-averaging).

\section{Further questions}\label{SecFurtherQuestions}
Later investigations in mind, we state a few further questions: 
\begin{enumerate}
\item The self-similar scaling behavior of the persistence diagram expectation measure of a time-dependent Poisson point process and of scaling function sublevel sets has been derived without using the packing relation, though confirming the latter. 
While the packing relation as derived in this work only holds for filtrations of complexes with persistent homology groups isomorphic to those of the \v{C}ech complex filtration, the derivation for the Poisson point process holds more generally. 
This raises the question if the packing relation can be extended to more general situations. 
A related question is if the ergodicity assumption, which is crucial at various points in this work, can be relaxed.
\vs
\item In the physics application we describe temporally self-similar scaling of persistence diagram expectation measures has been found. 
In this case, however, also the correlation functions of the point clouds show self-similar scaling~\cite{spitz2020finding}. 
Do there exist other parameter-dependent point processes, for which persistence diagram expectation measures show self-similar scaling but the correlation functions do not? 
This would provide further motivation for other applications in the natural sciences.
\vs
\item With respect to applications an extension of our results to filtrations of weighted simplicial complexes would be of interest. 
\end{enumerate}

\section*{Acknowledgements}
We cordially thank J. Berges, M. Oberthaler, H. Edelsbrunner, M. Bleher and M. Schmahl for discussions and collaborations on related work. 
We thank D. Kirchhoff for highlighting to us the idea of the example given in Sec. \ref{SecScalingFunctionSublevels}. 
This work is part of and supported by the Collaborative Research Centre, Project-ID No. 273811115, SFB 1225 ISOQUANT of the Deutsche Forschungsgemeinschaft (DFG, German Research Foundation), and supported by the Deutsche Forschungsgemeinschaft (DFG, German Research Foundation) under Germany's Excellence Strategy EXC 2181/1 - 390900948 (the Heidelberg STRUCTURES Excellence Cluster). 
AW acknowledges support from the Klaus Tschira Foundation.

\appendix

\section{Bounded total persistence}\label{AppendixBoundedTotalPersistence}
In this appendix we discuss the concept of bounded total persistence, introduced in~\cite{Cohen-Steiner2010}.
With $(M,d)$ a compact, triangulable metric space, let $f:M\to \rr$ be a Lipschitz function with Lipschitz constant $\lip(f) = \inf \{c\in \rr\,|\, |f(x)-f(y)|\leq c\, d(x,y)\,\forall\,x,y\in M\}$. 
We denote the $\ell$-th persistence diagram computed via singular homology from the filtration of sublevel sets of a Lipschitz function $f$ by $\dgm_\ell(f)$ and call $f$ \emph{tame}, if all its persistence diagrams are finite. 
Let $K$ be a finite simplicial complex triangulating $M$ via a triangulation homeomorphism $\vartheta$ and set $\textrm{mesh}(K):=\max_{\sigma \in K} \diam(\sigma)$, $\diam(\sigma):=\max_{x,y\in \sigma} d(\vartheta(x), \vartheta(y))$ the diameter of a simplex $\sigma$.  
For $r>0$ we define $N(r):=\min_{\mesh(K)\leq r} \# K$. 
We consider the total persistence diagram $\cup_\ell \dgm_\ell(f)$ of a Lipschitz function $f$, and define the \emph{degree-$k$ total persistence} of $f$ as
\begin{equation*}
\Pers_k(f):= \sum_{x\in \cup_\ell \dgm_\ell(f)} \pers(x)^k.
\end{equation*}
\vs

\begin{prop}[Technical results from \cite{Cohen-Steiner2010}]\label{PropPolynomialGrowth}
Assume that the size of the smallest triangulation of a triangulable, compact metric space $M$ \emph{grows polynomially} with one over the mesh, i.e., there exist $C_0,m$, such that $N(r)\leq C_0/r^m$ for all $r>0$. 
Let $\delta>0$ and $k=m+\delta$. Then,
\begin{equation*}
\Pers_k(f)\leq \frac{m+2\delta}{\delta} C_0  \, \lip(f)^m \amp(f)^\delta,
\end{equation*}
where $\amp(f):= \max_{x\in M} f(x) - \min_{y\in M} f(y)$ is the amplitude of $f$. 
\end{prop}
\vs

For $M$ a compact Riemannian $n$-manifold constants $c,C$ exist, such that $c/r^n \leq N(r) \leq C/r^n$ for sufficiently small $r$. 
A compact metric space $M$ \emph{implies bounded degree-$k$ total persistence}, if there exists $C_M>0$, such that $\Pers_k(f)\leq C_M$ for every tame function $f:M\to\rr$ with $\lip(f)\leq 1$.

\section{Functional summaries}
Persistence diagrams themselves do not naturally lead to statistical goals \cite{Mileyko2011}. 
Instead, functional summaries of persistence diagrams such as the well-known persistence landscapes~\cite{Bubenik2015} have been proposed for their statistical analysis. 
Moreover, in recent years a multitude of different functional summaries have been developed across the literature~\cite{Berry2018, biscio2019accumulated, chazal2014stochastic, chen2015statistical}.

As a more general framework, in this appendix we consider functional summaries as introduced in \cite{Berry2018}. 
Different types of functional summaries accompanied by corresponding limit theorems then lead to the proofs of Propositions \ref{PropExpectationNumberMaxDeath} and~\ref{PropExpectationDegreeQPersistenceLength}.

\subsection{Additive functional summaries}\label{SecAdditiveFunctionalSummaries}
Let $T$ be a compact metric space and $\mathscr{F}(T)$ a collection of functions, $f:T\to \rr$. 
A functional summary $\mathcal{F}$ is a map $\mathcal{F}:\mathscr{D}\to\mathscr{F}(T)$. 
We call a functional summary $\mathcal{F}$ \emph{additive}, if for any two persistence diagrams $D,E\in  \mathscr{D}$ with $D+E$ defined as the union of multisets the equation $\mathcal{F}(D+E) = \mathcal{F}(D)+\mathcal{F}(E)$ is fulfilled.%
\footnote{In~\cite{chazal2018density, divol2019choice} additive functional summaries have been introduced similarly as so-called linear representations of persistence diagrams.} 
We denote the set of additive functional summaries by $\mathscr{A}(T)$.

A functional summary $\mathcal{F}$ is \emph{uniformly bounded}, if a constant $U<\infty$ exists, such that
\begin{equation*}
\sup_{f\in \im (\mathcal{F})} \sup_{s\in T} |f(s)| \leq U.
\end{equation*}
The following proposition on the pointwise convergence of uniformly bounded functional summaries has been given in the literature \cite{Berry2018}.
\vs

\begin{prop}[Pointwise convergence of functional summaries \cite{Berry2018}]\label{PropPointwiseConvergence}
Let $\mathcal{F}$ be a uniformly bounded functional summary and $D_i\in \mathscr{D}$ for $i\in\nn$, sampled i.i.d. from a probability space $(\mathscr{D}, \mathcal{B}(\mathscr{D}), \mathbb{P}_{\mathscr{D}})$. Set $f_i:=\mathcal{F}(D_i)$. If $\im(\mathcal{F})$ is equicontinuous, then a.s. for $m\to \infty$
\begin{equation*}
\sup_{s\in T} \left|\frac{1}{m}\sum_{i=1}^m f_i(s) - \mathbb{E}[\mathcal{F}(D)(s)]\right| \to 0.
\end{equation*}
\end{prop}
\vs

In the framework of functional summaries persistence diagram (expectation) measures naturally show up.
\vs

\begin{prop}
Let $\mathcal{A}\in \mathscr{A}(T)$ be an additive functional summary, $s\in T$ and $A\in \mathcal{B}^n_b$. Let $Y_\omega(A)\subset\Delta$ be the multiset of atoms of $\rho_\omega(A)$. In the evaluation of $\mathcal{A}$ the persistence diagram measure $\rho$ appears,
\begin{equation*}
\mathcal{A}(Y_\omega(A))(s) = \sum_{x\in Y_\omega(A)} \mathcal{A}(\{x\})(s) = \int_{\Delta} \mathcal{A}(\{x\})(s) \, \rho_\omega(A)(\dd x).
\end{equation*}
\end{prop}

\begin{proof}
The statement is clear.
\end{proof}
\vs

\begin{prop}\label{PropExpectationValueFuncSummary}
Assume that the persistence diagram expectation measure $\mathfrak{p}$ exists. Then, for any functional summary $\mathcal{F}\in \mathscr{F}(T)$, $s\in T$ and $A\in\mathcal{B}^n_b$, we have 
\begin{equation*}
\mathbb{E}\left[ \int_{\Delta} \mathcal{F}(\{x\})(s) \, \rho_\omega(A)(\dd x)\right] = \int_{\Delta} \mathcal{F}(\{x\})(s)\,  \mathfrak{p}(A)(\dd x).
\end{equation*}
\end{prop}
\vs

\begin{proof}
This statement directly follows from the theory of point processes and their moment measures~\cite{daley2007}. 
Note that for any $A\in\mathcal{B}^n_b$ there exists a bounded $B\subset \overline{\Delta}$, such that $\supp(\rho_\omega(A))\subseteq B$ for any $\omega\in\Omega$.
Then 
\begin{equation*}
\int_{\Delta} \mathcal{F}(\{x\})(s) \, \rho_\omega(A)(\dd x) = \int_{\overline{\Delta}} \mathcal{F}(\{x\})(s) \chi_B(x) \, \rho_\omega(A)(\dd x)
\end{equation*}
with indicator function $\chi_B$, and $\mathcal{F}(\cdot)(s) \chi_B$ has bounded support.
\end{proof}

\subsection{Intensive functional summaries}\label{SecIntensiveFunctionalSummaries}
\begin{defn}
Let $\xi$ be a stationary and ergodic simple point process on $\rr^n$ and $\mathcal{F}:\mathscr{D}\to\mathscr{F}(T)$ a functional summary. 
We say that $\mathcal{F}$ is \emph{$\xi$-intensive}, if for any balanced convex averaging sequence $\{A_k\}$, $\epsilon>0$ and $k$ sufficiently large we a.s. have
\begin{equation*}
 \lim_{l\to \infty}|| \mathcal{F}(D_l) - \mathcal{F}(D_k)||_\infty < \epsilon,
\end{equation*}
where $||\cdot ||_\infty$ denotes the supremum norm and  $D_k:= \bigcup_{\ell=0}^{n-1} \dgm_\ell (X_{\xi_\omega} (A_k))$.
\end{defn}
\vs

\begin{lem}\label{LemmaAdditiveDivideVolumeIntensiveFuncSummary}
Let $\mathcal{A}\in \mathscr{A}(T)$ be an additive functional summary and $\xi$ a stationary and ergodic simple point process on $\rr^n$. 
Set $D_k:= \bigcup_{\ell=0}^{n-1} \dgm_\ell (X_{\xi_\omega}(A_k))$, where $\{A_k\}$ is a balanced convex averaging sequence. 
Then $\mathcal{A}(D_k)/\lambda_n(A_k)$ a.s. defines a $\xi$-intensive functional summary.
\end{lem}

\begin{proof}
Set $X_k:=X_{\xi_\omega} (A_k)$. 
Let $s\in T$ and $k,l\in \nn$,
\begin{align}
&\left| \frac{ \mathcal{A}(D_l)(s)}{\lambda_n(A_l)} - \frac{\mathcal{A}(D_k)(s)}{\lambda_n(A_k)}\right|\nonumber\\
&\quad =\, \left| \frac{\mathcal{A}(D_l)(s)}{n(X_l)} \left(\frac{n(X_l)}{\lambda_n(A_l)} -  \frac{n(X_k)}{\lambda_n(A_k)}\right)  - \frac{n(X_k)}{\lambda_n(A_k)}\left( \frac{\mathcal{A}(D_k)(s)}{n(X_k)}-\frac{\mathcal{A}(D_l)(s)}{n(X_l)}\right)\right|\nonumber\\
&\quad \leq \,  \frac{|\mathcal{A}(D_l)(s)|}{n(X_l)} \left| \frac{n(X_l)}{\lambda_n(A_l)} - \frac{n(X_k)}{\lambda_n(A_k)} \right| + \frac{n(X_k)}{\lambda_n(A_k)}\left| \frac{\mathcal{A}(D_k)(s)}{n(X_k)}-\frac{\mathcal{A}(D_l)(s)}{n(X_l)}\right|.\label{EqEstimationDeviations}
\end{align}
By additivity of $\mathcal{A}$ we obtain that $\mathcal{A}(D_k)(s)/n(X_k)$ converges to some constant \mbox{$a(s)<\infty$} for $k\to\infty$ and every~$s$. 
To this end, for sufficiently large $k,l$:
\begin{equation*}
\left| \frac{\mathcal{A}(D_k)(s)}{n(X_k)} - \frac{\mathcal{A}(D_l)(s)}{n(X_l)}\right| < \epsilon.
\end{equation*}
The point process $\xi$ is ergodic in persistence by \Cref{LemmaErgodicPersistence}, such that a.s. for $\epsilon >0$ and sufficiently large~$k,l$:
\begin{equation*}
\left|\frac{n(X_l)}{\lambda_n(A_l)} - \frac{n(X_k)}{\lambda_n(A_k)}\right| < \epsilon.
\end{equation*}
Furthermore, by \Cref{ThmExistenceLimitingRadonMeasure} a.s. $c^*>0$ exists, such that for all $k$: \mbox{$n(X_k)/\lambda_n(A_k) < c^*$}. 
Insertion of all this into \Cref{EqEstimationDeviations} a.s. yields for sufficiently large $k,l$,
\begin{equation*}
\left| \frac{ \mathcal{A}(D_l)(s)}{\lambda_n(A_l)} - \frac{\mathcal{A}(D_k)(s)}{\lambda_n(A_k)}\right| < (a(s)+c^*) \epsilon.
\end{equation*}
Indeed, $\mathcal{A}(D_k)/\lambda_n(A_k)$ a.s. defines a $\xi$-intensive functional summary.
\end{proof}
\vs

\begin{cor}\label{CorAdditiveDivideNumberIntensiveFuncSummary}
Given the assumptions of \Cref{LemmaAdditiveDivideVolumeIntensiveFuncSummary}, $\mathcal{A}(D_k) / n(D_k)$ is a $\xi$-intensive functional summary. 
This is a direct result of
\begin{equation*}
\frac{\mathcal{A}(D_k)}{n(D_k)} = \frac{\lambda_n(A_k)}{n(D_k)} \frac{\mathcal{A}(D_k)}{\lambda_n(A_k)},
\end{equation*}
both factors being $\xi$-intensive functional summaries.
\end{cor}
\vs

Asymptotically, for $\xi$-intensive functional summaries the ensemble-average can be replaced by the large-volume limit.
\vs

\begin{prop}\label{PropIntensiveFuncSummariesVolEnsembleAvg}
Let $\xi$ be a stationary and ergodic simple point process on $\rr^n$. 
Let $\mathcal{F}$ be a $\xi$-intensive functional summary, which is uniformly bounded and equicontinuous, and let $\{A_k\}$ be a balanced convex averaging sequence. 
Choose i.i.d. samples $\omega_i\in \Omega$, $i\in \nn$, according to the probability distribution $\mathbb{P}$ and define $D_{k,i}:=\bigcup_{\ell=0}^{n-1} \dgm_\ell (X_{\xi_{\omega_i}}(A_k))$. 
Then for any $j\in\nn$, $\epsilon>0$ and sufficiently large $k$ we a.s. find
\begin{equation*}
\sup_{s\in T} \left| \lim_{l\to \infty} \mathcal{F}(D_{l,j})(s) - \lim_{m\to\infty} \frac{1}{m} \sum_{i=1}^m \mathcal{F}(D_{k,i})(s)\right| < \epsilon.
\end{equation*}
\end{prop}

\begin{proof}
Let $s\in T$, $\epsilon>0$, $j\in\nn$, $D_k:=\bigcup_{\ell=0}^{n-1}\dgm_\ell(X_\xi(A_k))$. 
Since $\mathcal{F}$ is $\xi$-intensive, further exploiting \Cref{PropExpectationValueFuncSummary} and \Cref{ThmExistenceLimitingRadonMeasure}, we a.s. obtain for sufficiently large~$k$
\begin{equation*}
\left| \lim_{l\to\infty} \mathcal{F}(D_{l,j})(s) - \mathbb{E}[\mathcal{F}(D_{k})(s)]\right| < \frac{\epsilon}{2}.
\end{equation*}
Similarly, \Cref{PropPointwiseConvergence} a.s. yields for sufficiently large $k$
\begin{equation*}
\left| \mathbb{E}[\mathcal{F}(D_{k})(s)] - \lim_{m\to \infty}\frac{1}{m} \sum_{i=1}^m \mathcal{F}(D_{k,i})(s) \right| < \frac{\epsilon}{2}.
\end{equation*}
Putting things together, we a.s. find
\begin{align*}
&\left| \lim_{l\to \infty} \mathcal{F}(D_{l,j})(s) - \lim_{m\to\infty} \frac{1}{m} \sum_{i=1}^m \mathcal{F}(D_{k,i})(s)\right| \nonumber\\
&\quad \leq \, \left| \lim_{l\to\infty} \mathcal{F}(D_{l,j})(s) - \mathbb{E}[\mathcal{F}(D_{k})(s)]\right| + \left| \mathbb{E}[\mathcal{F}(D_{k})(s)] - \lim_{m\to \infty}\frac{1}{m} \sum_{i=1}^m \mathcal{F}(D_{k,i})(s) \right| < \epsilon.
\end{align*}
\end{proof}

\subsection{Proofs of Propositions \ref{PropExpectationNumberMaxDeath} and \ref{PropExpectationDegreeQPersistenceLength}}\label{AppendixProofsPropsGeometricQuantities}
Given the preceding statements on functional summaries, we can now prove Propositions \ref{PropExpectationNumberMaxDeath} and \ref{PropExpectationDegreeQPersistenceLength}.

\begin{proof}[Proof of \Cref{PropExpectationNumberMaxDeath}]
The statement directly follows from \Cref{PropExpectationValueFuncSummary}.
\end{proof}

\begin{proof}[Proof of \Cref{PropExpectationDegreeQPersistenceLength}]
The final statement on finiteness of $\mathfrak{d}_{\max}(A)$ and $\mathfrak{l}_q(A)$ is clear for any $A\in\mathcal{B}^n_b$.
Further, for any $\omega\in\Omega$ and $A\in\mathcal{B}^n_b$ we define
\begin{equation*}
L_{q,\omega}(A) := \int_{\Delta} \pers(x)^q \rho_\omega(A)(\dd x).
\end{equation*}
$L_q$ is an additive functional summary. 
Under the assumptions of the proposition \Cref{CorAdditiveDivideNumberIntensiveFuncSummary} holds, such that $L_{q,\omega}(A)/n_\omega(A)$ makes up a $\xi$-intensive functional summary. 
Let $\{A_k\}$ be a balanced convex averaging sequence and $\omega_i\in\Omega$ i.i.d. for $i\in\nn$. 
For any $j\in\nn$ we a.s. obtain by means of \Cref{PropPointwiseConvergence}, \Cref{CorAdditiveDivideNumberIntensiveFuncSummary} and \Cref{PropIntensiveFuncSummariesVolEnsembleAvg} for $\epsilon>0$ and sufficiently large $k$:
\begin{equation*}
\left|\mathbb{E}[l_{q,\omega}(A_k)^q] -\lim_{k'\to\infty} l_{q,\omega_j}(A_{k'})^q \right| = \left|\lim_{m\to\infty}\frac{1}{m} \sum_{i=1}^m l_{q,\omega_i}(A_k)^q - \lim_{k'\to\infty} l_{q,\omega_j}(A_{k'})^q\right|<\epsilon.
\end{equation*}
Now,
\begin{equation*}
l_{q,\omega_j}(A_{k'})^q = \frac{L_{q,\omega_j}(A_{k'})}{n_{\omega_j}(A_{k'})} = \frac{\lambda_n(A_{k'})}{n_{\omega_j}(A_{k'})} \frac{L_{q,\omega_j}(A_{k'})}{\lambda_n(A_{k'})}.
\end{equation*}
Each of the two factors on the right-hand side is a $\xi$-intensive functional summary according to \Cref{LemmaAdditiveDivideVolumeIntensiveFuncSummary}. 
We a.s. get for sufficiently large $k$
\begin{equation*}
\frac{\epsilon}{2}>\left|\mathbb{E}[l_{q,\omega}(A_k)^q]^{1/q} - \left(\lim_{k'\to\infty} \frac{\lambda_n(A_{k'})}{n_{\omega_j}(A_{k'})} \lim_{k'\to\infty} \frac{L_{q,\omega_j}(A_{k'})}{\lambda_n(A_{k'})}\right)^{1/q}\right|.
\end{equation*}
Thus, a.s. for sufficiently large $k$:
\begin{subequations}
\begin{align*}
\epsilon>&\,   \left|\mathbb{E}[l_{q,\omega}(A_k)^q]^{1/q}- \left(\frac{\lambda_n(A_k)}{\lim_{m\to\infty}\frac{1}{m} \sum_{i=1}^m n_{\omega_i}(A_k)}\lim_{m\to\infty}\frac{1}{m} \sum_{i=1}^m \frac{L_{q,\omega_i}(A_k)}{\lambda_n(A_k)}\right)^{1/q}\right| + \frac{\epsilon}{2}\\
>&\, \left|\mathbb{E}[l_{q,\omega}(A_k)^q]^{1/q} - \left(\frac{\lambda_n(A_k)}{\mathbb{E}[n_{\omega_i}(A_k)]} \mathbb{E}\left[ \frac{L_{q,\omega_i}(A_k)}{\lambda_n(A_k)}\right]\right)^{1/q}\right|\\
=&\, \left|\mathbb{E}[l_{q,\omega}(A_k)^q]^{1/q} - \mathfrak{l}_q(A_k)\right|.
\end{align*}
\end{subequations}
Finally,
\begin{subequations}
\begin{align*}
\epsilon>\,\left|\mathbb{E}[l_{q,\omega}(A_k)^q]^{1/q}  - \left(\lim_{k'\to\infty} l_{q,\omega_j}(A_{k'})^q\right)^{1/q}\right| =&\, \left|\mathbb{E}[l_{q,\omega}(A_k)^q]^{1/q}  - \lim_{k'\to\infty} l_{q,\omega_j}(A_{k'})\right|\\
=&\,\left|\mathbb{E}[l_{q,\omega}(A_k)^q]^{1/q}  - \mathbb{E}[l_{q,\omega}(A_k)]\right|,
\end{align*}
\end{subequations}
completing the proof for $\mathfrak{l}_q(A_k)$. 

The proof of $|\mathfrak{d}_{\max}(A) - \mathbb{E}[d_{\max,\omega}(A)]|<\epsilon$ works similarly. 
Defining for all $\mathcal{B}_b^n$ and $q>0$
\begin{equation*}
d_{q,\omega}(A):=\int_{\Delta}d(x)^q \rho_\omega(A)(\dd x),
\end{equation*}
we note that via \Cref{EqMaximumDeathDef}
\begin{equation*}
\lim_{q\to\infty} \left(\frac{d_{q,\omega}(A)}{n_{\omega}(A)}\right)^{1/q} = d_{\max,\omega}(A).
\end{equation*}
The quantity $d_{q,\omega}(A)/n_{\omega}(A)$ defines a $\xi$-intensive functional summary. 
Then arguments analogously to those for $\mathfrak{l}_q(A)$ together with monotone convergence lead for sufficiently large $k$ to
\begin{equation*}
|\mathfrak{d}_{\max}(A_k)-\mathbb{E}[d_{\max,\omega}(A_k)]| < \epsilon.
\end{equation*}

Let $p,q\neq 0$. 
Then, $d_{\max,\omega}(A_k)^p/l_{q,\omega}(A_k)^q$ constitutes a $\xi$-intensive functional summary. 
As such, Propositions \ref{PropPointwiseConvergence} and  \ref{PropIntensiveFuncSummariesVolEnsembleAvg} apply, yielding \Cref{EqErgodicConvergenceFracdmaxlq}.
\end{proof}

\section{Proofs of limit theorems for balanced convex averaging sequences}\label{AppendixProofsLimitTheorems}
In this appendix we provide the proofs for \Cref{ThmExistenceLimitingRadonMeasure} and \Cref{CorStrongLawPersBetti}. 
First, relevant results from convex geometry are revisited.
\vs

\begin{prop}[Steiner's formula~\cite{gruber2007}]\label{PropSteinerFormula}
Let $C$ be a convex body in $\rr^n$. 
Then, for $\delta \geq 0$,
\begin{equation}\label{EqSteinersFormula}
\lambda_n(C \oplus  \delta  B_1(0)) = \lambda_n(C) + \sum_{i=1}^n \binom{n}{i} W_i(C) \delta^i,
\end{equation}
where  $W_i(C)$ are the quermassintegrals, given in \Cref{EqQuermassintegralDef}.
\end{prop}
\vs

Using Aleksandrov-Fenchel inequalities from convex geometry in terms of the quermassintegrals, an upper bound for the quermassintegrals can be given.
\vs

\begin{prop}[Theorem 2 in~\cite{mcmullen1991inequalities} rephrased for quermassintegrals]\label{PropAleksandrovFenchel}
Let $K\subset \rr^n$ be a convex body. 
Then, for any $0\leq j \leq n$
\begin{equation*}
    W_{n-j}(K)\leq \frac{n^j \kappa_{n-j}}{j! \binom{n}{j} \kappa_{n-1}^j} W_{n-1}(K)^{j},
\end{equation*}
with $\kappa_j$ the volume of a $j$-dimensional unit ball.
\end{prop}

\begin{proof}[Proof of \Cref{ThmExistenceLimitingRadonMeasure}]
We largely follow the strategy of \cite{hiraoka2018}. 
The key step is that we construct a tessellation by cubes for each convex set $A_k$, which allows us to prove a statement on the convergence of persistent Betti numbers of point clouds in the sets~$A_k$.
The lengthy proof proceeds in three steps. 
First, we employ Theorem 1.11 of \cite{hiraoka2018} to obtain persistent Betti numbers for large cubic volumes in $\rr^n$. 
Second, we construct a tessellation by cubes for each $A_k$ and use methods from convex geometry to obtain a statement on the convergence of persistent Betti numbers of point clouds in the sets $A_k$. 
Third, we assemble the results to obtain the desired statement for measures similarly to the derivation of Theorem 1.5 in \cite{hiraoka2018}.

\emph{First step.} Let $\beta_\ell^{r,s}(\mathcal{C}(X))$ be the $\ell$-th persistent Betti numbers of the \v{C}ech complex filtration of a point cloud $X\subset \rr^n$. 
We define for any $A\in\mathcal{B}^n_b$,
\begin{equation*}
\psi(A)=\mathbb{E}[\beta_\ell^{r,s}(\mathcal{C}(X_{\xi_\omega}(A)))]
\end{equation*}
for some $r\leq s$. 
For each $L>0$ let $\Lambda_L:=[0,L]^n$ be the cube of side length $L$. 
For any $A_k$ in the convex averaging sequence there exists a unique $L_k'$, such that $\lambda_n(A_k) = \lambda_n(\Lambda_{L_k'}) = (L_k')^n$. 
We set 
\begin{equation*}
L_k:=\sup\{ L \, |\, \exists \, x\in \rr^n: \Lambda_L+x \subseteq A_k\}.
\end{equation*}
Note that by construction $L_k\leq L_k'$, and $\{\Lambda_{L_k}\, |\, k\}$ as well as $\{\Lambda_{L_k'}\, |\, k\}$ are convex averaging sequences, too.
Given $\epsilon>0$, by means of Theorem 1.11 of \cite{hiraoka2018} we find for sufficiently large $k,m$:
\begin{equation}\label{EqConvergenceCubes}
\left| \frac{\psi(\Lambda_{L_k'})}{(L_k')^n} -  \frac{\psi(\Lambda_{L_m})}{L_m^n}\right|<\epsilon.
\end{equation}

\emph{Second step.} We want to show that for sufficiently large $k,m$:
\begin{equation}\label{EqAkLmCorrespVolumes}
\left| \frac{\psi(A_k)}{\lambda_n(A_k)} - \frac{\psi(\Lambda_{L_m})}{L_m^n} \right| <\epsilon.
\end{equation}
In order to establish this, we investigate the scaling behavior of quermassintegrals $W_i(A_k)$ with the volume of $A_k$. 
For this purpose, we construct a tessellation made up from cubes. Using the $n$-dimensional lattice
\begin{equation*}
L_M \cdot \zz^n = \{(L_M z_1,\dots, L_M z_n)\, |\, z_i\in \zz\}
\end{equation*}
with $0<M\in\nn$ arbitrary, we set
\begin{equation*}
\Lambda(A_k) := \bigcup_{x\in (L_M\cdot \zz^n) \cap A_k} \big(\Lambda_{L_M} + x\big).
\end{equation*}
By boundedness of each $A_k$, $(L_M\cdot \zz^n) \cap A_k$ is finite for all $k$. 
If we already find $A_k\subseteq \Lambda(A_k)$, we define $C_k := \Lambda(A_k)$. 
If on the other hand $B_k := A_k \setminus \Lambda(A_k)\neq \emptyset$, we choose $x_1\in B_k$. 
For $x_1$ there exist finitely many $x'_i\in L_M\cdot \zz^n$, $i=1,\dots,N'$, such that $x_1\in \Lambda_{L_M}+x_i'$ for all $i=1,\dots,N'$. 
We define
\begin{equation*}
\Lambda_1(A_k):=\Lambda(A_k)\cup \bigcup_{i=1}^{N'}(\Lambda_{L_M} + x'_i),\quad B_{k,1}:= A_k\setminus \Lambda_1(A_k).
\end{equation*}
By boundedness of $A_k$, we can inductively repeat this construction finitely many times, until $B_{k,N} :=A_k \setminus\Lambda_N(A_k)=\emptyset$, i.e., $A_k\subseteq \Lambda_N(A_k)$. 
Then we set $C_k:=\Lambda_N(A_k)$. 
The $C_k$ make up the desired tessellation.
Since $A_k\subseteq C_k$, it follows by monotonicity of the quermassintegrals \cite{gruber2007} that
\begin{equation}\label{EqMonotonicityQuermassintegrals}
W_i(A_k) \leq W_i(C_k). 
\end{equation}
By construction, we find $C_k\subseteq A_k \oplus  \sqrt{n}L_M B_1(0)$. 
Thus, using \Cref{EqMonotonicityQuermassintegrals} and Steiner's formula (\Cref{PropSteinerFormula}) we obtain
\begin{equation}\label{EqSteinerFormulaEmployed}
\lambda_n(C_k)\leq \lambda_n(A_k \oplus  \sqrt{n}L_M B_1(0)) = \lambda_n(A_k) + \sum_{i=1}^n \binom{n}{i} (\sqrt{n} L_M)^i W_i(A_k).
\end{equation}
By \Cref{PropAleksandrovFenchel} constants $\kappa_{n,j} > 0$ exist, such that for $0\leq j \leq n$: $W_{n-j}(A_k)\leq \kappa_{n,j} W_{n-1}(A_k)^{j}$.
Exploiting that we restrict to balanced sequences $\{A_k\}$ with their $(n-1)$-th quermassintegrals scaling as $W_{n-1}(A_k)=O(\lambda_n(A_k)^{1/n}) = O(L_k')$ for large $k$, we thus find for $0\leq j \leq n$ that
\begin{equation}\label{EqScalingWnjAk}
    W_{n-j}(A_k) = O((L_k')^{j}).
\end{equation}
Using \Cref{EqSteinerFormulaEmployed,EqScalingWnjAk}, we obtain an estimate for the volume difference between $C_k$ and $A_k$,
\begin{equation*}
\left|\lambda_n(C_k)-\lambda_n(A_k)\right| \leq  \left| \sum_{i=1}^n \binom{n}{i} W_i(A_k) (\sqrt{n}L_M)^i\right|= O((L_k')^{n-1} L_M).
\end{equation*}
Hence,
\begin{equation}\label{EqVolumeConvergenceTessellation}
\left| \frac{\lambda_n(C_k)}{\lambda_n(A_k)} - 1\right| = O\big(L_M (L_k')^{-1}\big) \to 0 \qquad \textrm{as } k\to \infty.
\end{equation}
For later use we state the following 
\begin{equation}\label{EqCkLmConvergence}
\left| \frac{\psi(C_k)}{\lambda_n(A_k)}- \frac{\psi(\Lambda_{L_m})}{L_m^n}  \right| \leq \left| \left(\frac{\lambda_n(C_k)}{\lambda_n(A_k)}-1\right) \frac{\psi(C_k)}{\lambda_n(C_k)}\right| + \left| \frac{\psi(C_k)}{\lambda_n(C_k)} - \frac{\psi(\Lambda_{L_m})}{L_m^n}\right|.
\end{equation}
There exists $K_1\in\nn$, such that for all $k\geq K_1$  the first term is bounded by $\epsilon/2$ (due to \Cref{EqVolumeConvergenceTessellation}). 
Applying the same argument as in the proof of Theorem 1.11 of~\cite{hiraoka2018} for the second term, which is applicable since only cubes appear in it, there exists a $K_2\in\nn$, such that the second term is bounded from above by $\epsilon/2$ for all $k\geq K_2$.
Thus, the left-hand side of \Cref{EqCkLmConvergence} converges to zero for $k,m\to\infty$.

\Cref{EqCkLmConvergence} provides the necessary inequality to continue with arguments of the proof of Theorem 1.11 of~\cite{hiraoka2018}. 
Starting from
\begin{equation}\label{EqpersBettiBoundNSimplices}
|\beta_\ell^{r,s}(\mathcal{C}(X_{\xi_\omega}(C_k))) - \beta_\ell^{r,s}(\mathcal{C}(X_{\xi_\omega}(A_k)))| \leq \sum_{j=\ell}^{\ell+1} F_j(\xi,s; C_k\setminus A_k)),
\end{equation}
where again $F_j(\xi, r;A)$ is the number of $j$-simplices in $\check{C}_r(X_\xi(\rr^n))$ with at least one vertex in $A$ (see the proof of \Cref{LemmaExistencePersDiagExpecMeasure}), we have 
\begin{equation}\label{EqScalingFj}
\mathbb{E}[F_j(\xi,s;C_k\setminus A_k))] = O(\lambda_n(C_k\setminus A_k)).
\end{equation}
Thus, with \Cref{EqVolumeConvergenceTessellation} we obtain
\begin{equation}\label{EqCkAkConvergence}
\left| \frac{\psi(C_k)}{\lambda_n(A_k)} - \frac{\psi(A_k)}{\lambda_n(A_k)}\right| = O(\lambda_n(C_k\setminus A_k) \lambda_n(A_k)^{-1})\to 0\qquad \textrm{as } k\to\infty.
\end{equation}
Assembling \Cref{EqCkLmConvergence,EqCkAkConvergence}, we finally get 
\begin{equation*}
\left| \frac{\psi(A_k)}{\lambda_n(A_k)} - \frac{\psi(\Lambda_{L_m})}{L_m^n} \right|\leq  \left| \frac{\psi(C_k)}{\lambda_n(A_k)}- \frac{\psi(\Lambda_{L_m})}{L_m^n}  \right| + \left|\frac{\psi(C_k)}{\lambda_n(A_k)} -  \frac{\psi(A_k)}{\lambda_n(A_k)}\right| \to 0
\end{equation*}
for $k,m \to \infty$.
Therefore the desired inequality (\Cref{EqAkLmCorrespVolumes}) holds. 

\emph{Third step.} Using \Cref{EqConvergenceCubes,EqAkLmCorrespVolumes}, we find for sufficiently large $k,m$
\begin{equation*}
\left|\frac{\psi(\Lambda_{L_k'})}{(L_k')^n} - \frac{\psi(A_k)}{\lambda_n(A_k)}\right|\leq   \left| \frac{\psi(\Lambda_{L_k'})}{(L_k')^n} - \frac{\psi(\Lambda_{L_m})}{L_m^n} \right| + \left| \frac{\psi(\Lambda_{L_m})}{L_m^n}  - \frac{\psi(A_k)}{\lambda_n(A_k)}\right| <2\epsilon.
\end{equation*}
Hence, Theorem 1.11 of~\cite{hiraoka2018} holds for any balanced convex averaging sequence $\{A_k\}$ instead of $\{\Lambda_L\}$, that is, there exists a constant $\hat{\beta}_\ell^{r,s}$ such that
\begin{equation*}
\frac{\psi(A_k)}{\lambda_n(A_k)} \to \hat{\beta}^{r,s}_\ell\quad \textrm{ for }k\to\infty.
\end{equation*}
The proof of Theorem 1.5 in~\cite{hiraoka2018} holds now equally in this more general case. Indeed, on $\Delta$ a unique Radon measure $\mathfrak{P}$ exists with the property
\begin{equation*}
\frac{\mathbb{E}[\rho_\omega(A_k)]}{\lambda_n(A_k)} \overset{v}{\longrightarrow}\mathfrak{P} \quad  \textrm{ for } k \to \infty.
\end{equation*}
\end{proof}

We now deliver the proof of the strong law for persistent Betti numbers for balanced convex averaging sequences.
\begin{proof}[Proof of \Cref{CorStrongLawPersBetti}]
The first statement follows directly from the proof of \Cref{ThmExistenceLimitingRadonMeasure}. 
For the second statement, we construct a tessellation $C_k$ for each of the $A_k$ as in the proof of \Cref{ThmExistenceLimitingRadonMeasure}. 
Then we get 
\begin{align}
\left| \frac{\beta^{r,s}_\ell(\mathcal{C}(X_k))}{\lambda_n(A_k)} - \frac{\psi(A_k)}{\lambda_n(A_k)}\right| \leq &\, \left|  \frac{\beta^{r,s}_\ell(\mathcal{C}(X_k))}{\lambda_n(A_k)} -  \frac{\beta^{r,s}_\ell(\mathcal{C}(X_\xi(C_k)))}{\lambda_n(A_k)}\right|\nonumber\\
& + \left| \frac{\beta^{r,s}_\ell(\mathcal{C}(X_\xi(C_k)))}{\lambda_n(A_k)} - \frac{\psi(C_k)}{\lambda_n(A_k)} \right| + \left| \frac{\psi(C_k)}{\lambda_n(A_k)} - \frac{\psi(A_k)}{\lambda_n(A_k)}\right|.\label{EqPersistentBetaErgodicEstimation}
\end{align}
We estimate the first term on the right-hand side via \Cref{EqpersBettiBoundNSimplices}:
\begin{equation*}
\left|  \frac{\beta^{r,s}_\ell(\mathcal{C}(X_k))}{\lambda_n(A_k)} -  \frac{\beta^{r,s}_\ell(\mathcal{C}(X_\xi(C_k)))}{\lambda_n(A_k)}\right| \leq \sum_{j=\ell}^{\ell+1} \frac{F_j(\xi,s;C_k\setminus A_k)}{\lambda_n(A_k)}.
\end{equation*}
By means of ergodicity and \Cref{EqScalingFj} we obtain for sufficiently large $k$
\begin{equation*}
F_j(\xi,s;C_k\setminus A_k) \to \mathbb{E}\left[F_j(\xi,s;C_k\setminus A_k)  \right] = O(\lambda_n(C_k\setminus A_k)).
\end{equation*}
Thus, the first term on the right-hand side of \Cref{EqPersistentBetaErgodicEstimation} converges to zero for $k\to\infty$. 
The $C_k$ being made up from cubes which intersect only at mutual boundaries, the subsequent second term converges to zero for $k\to\infty$ by means of ergodicity (see the proof of Theorem 1.11 in~\cite{hiraoka2018}). 
The third term converges to zero for $k\to\infty$ due to \Cref{EqCkAkConvergence}.
\end{proof}

\addcontentsline{toc}{section}{Bibliography}
\bibliography{literature}
\bibliographystyle{halpha_altered}

\end{document}